\documentclass{article}
\usepackage{url}
\usepackage{todonotes}
\usepackage{amsthm, amsmath, amssymb, float, doi, subcaption}
\usepackage[capitalise]{cleveref}
\usepackage{fullpage}

\newtheorem{theorem}{Theorem}[section]
\newtheorem{lemma}[theorem]{Lemma}

\newtheorem{conjecture}[theorem]{Conjecture}
\newtheorem{proposition}[theorem]{Proposition}

\newtheorem{construction}[]{Construction}

\newtheorem{innercustomgeneric}{\customgenericname}
\providecommand{\customgenericname}{}
\newcommand{\newcustomtheorem}[2]{%
  \newenvironment{#1}[1]
  {%
   \ifdefined\crefalias\crefalias{innercustomgeneric}{#2}\fi
   \renewcommand\customgenericname{#2}%
   \renewcommand\theinnercustomgeneric{##1}%
   \innercustomgeneric
  }
  {\endinnercustomgeneric}%
  \ifdefined\crefname\crefname{#2}{#2}{#2s}\fi
}

\newcustomtheorem{customthm}{Theorem}

\newcommand{\setgivensymbol}[0]{:}

\newcommand{\obs}[2]{\operatorname{Obs}\left(#1;#2\right)}
\newcommand{\obsq}[2]{\operatorname{Obs}\left(#1;#2,q\right)}
\newcommand{\sizeobs}[2]{|\operatorname{Obs}\left(#1;#2\right)|}

\newcommand{\fork}[1]{\forka{#1}{ }}
\newcommand{\forka}[2]{\Psi_{#1}^{#2}}
\newcommand{\spoon}[1]{\spoona{#1}{ }}
\newcommand{\spoona}[2]{\Phi_{#1}^{#2}}

\newcommand{\gadgeta}[2]{R_{#1}^{#2}}
\newcommand{\glueg}[2]{\glue{G}{#1}{#2}}
\newcommand{\glue}[3]{{#1}\boxminus_{{#2}}{#3}}

\newcommand{\expol}[3]{\mathcal{E}\left(#1;#2,#3\right)}
\newcommand{\expolq}[2]{\expol{#1}{#2}{q}}

\tikzstyle{Meta Node}=[fill=none, draw=none, shape=circle]
\tikzstyle{Text Node}=[draw=none, fill=none, anchor=center, shape=rectangle, inner sep=0pt, outer sep=0pt, text height=9pt]
\tikzstyle{Phantom Node}=[draw=none, fill=none, shape=rectangle, inner sep=0pt, outer sep=0pt, text height=0pt]
\tikzstyle{Black Node}=[fill=white, draw=black, shape=circle]
\tikzstyle{Blue Node}=[fill=blue, draw=black, shape=circle]
\tikzstyle{Red Node}=[fill=red, draw=black, shape=circle]
\tikzstyle{Yellow Node}=[fill=yellow, draw=black, shape=circle]
\tikzstyle{Black Edge}=[-]
\tikzstyle{Dashed Edge}=[dashed]
\usetikzlibrary{graphs.standard}

\title{On Fragile Power Domination}
\author{Beth Bjorkman\footnote{Air Force Research Laboratory} \and Sean English\footnote{University of North Carolina Wilmington, \texttt{EnglishS@uncw.edu}} \and Johnathan Koch\footnote{Applied Research Solutions} \and Amanda Verga\footnote{Trinity College, ATRC}}
\date{\today}

\begin{document}

\maketitle

\begin{abstract}
    Power domination is a graph theoretic model which captures how phasor measurement units (PMUs) can be used to monitor a power grid. Fragile power domination takes into account the fact that PMUs may break or otherwise fail. In this model, each sensor fails independently with probability $q\in [0,1]$ and the surviving sensors monitor the grid according to classical power domination. 
    
    We study the expected number of observed nodes under the fragile power domination model. We give a characterization for when two networks and initial sensor placements will behave the same according to this expectation. We also show how to control the behavior of this expectation by adding structure to a network.
\end{abstract}
\section{Introduction}

    Power domination is a graph propagation process which models the observation of a power grid via \emph{phasor measurement units (PMUs)}. 
    As a graph process, power domination was originally defined by Brueni and Heath~\cite{BH2005} and then simplified by Haynes et al.~\cite{HHHH2002}. 
    Given a graph $G$, we choose some initial set $S\subseteq V(G)$ where PMUs are placed. 
    A PMU observes the vertex that it is placed on and all neighboring vertices. 
    This is called the \emph{domination step}.
    Conservation of energy laws are then applied; if there is an observed vertex with exactly one unobserved neighbor, the neighbor becomes observed.
    This is called the \emph{zero forcing step}.
    The domination step happens only once, while the zero forcing step occurs repeatedly until no new vertices are be observed.
    
    The most commonly studied question in power domination is the minimum number of PMUs needed so that the entire vertex set is observed when the propagation process finishes. 
    This is called the \emph{power domination number} of $G$, and is denoted $\gamma_P(G)$. 
    Power domination has received considerable attention in the literature recently (see e.g. \cite{AK2022, BFFFHV2018, DMKS2008, LMW2020, PAHA2024, YW2022}).
    
    A new trend in power domination research has been to take into account the fact that sensors may fail. 
    Pai, Chang, and Wang proposed \emph{$k$-fault tolerant power domination}~\cite{PCW2010}, which seeks to find sets which are power dominating even if any $k$ of the sensors are removed (fail). 
    A similar model, \emph{PMU-defect-robust power domination}~\cite{BCF2023}, proposed by Bjorkman, Conrad, and Flagg, also allows a fixed number of sensors to fail, however, this model allows for multiple sensors to be placed at the same vertex.
    
    PMU-defect-robust and $k$-fault tolerant power domination models both study the case where a known number of sensors fail and still seek to determine the minimum number of sensors required to observe the entire graph. 
    However, knowing exactly how many sensors will fail is not always a realistic assumption.
    Instead, \emph{fragile power domination}~\cite{BBFK2023}, proposed by Bjorkman, Brennan, Flagg, and Koch, considers the case where each sensor fails independently with probability $q$, followed by the standard power domination propagation process using the remaining set of sensors.
    In what follows, we will sometimes allow for PMU placements which are multisets as placing more than one sensor at a vertex increases the probability it will be observed.

    \subsection{Fragile Power Domination}
        As defined in \cite{HHHH2002}, the \emph{power domination process} on a graph $G$ with initial PMU placement $S$ is the following:
        \begin{enumerate}
            \item
                \textbf{(Domination Step)}: Let $B:= N[S]$.
            \item 
                \textbf{(Zero Forcing Step)}: While there exists some vertex $x \in B$ such that there is a single vertex, say $y$, in $N[x] \setminus B$, add $y$ to $B$. When this happens, we say $x$ \emph{forces} $y$.
        \end{enumerate}
        The set $B$ at the end of applying the power domination process is denoted $\obs{G}{S}$, and we say that vertices in $\obs{G}{S}$ are \emph{observed}.
        Vertices not in $\obs{G}{S}$ are correspondingly called unobserved.
        When the PMU placement is a singleton set $\{v\}$, we will write $\obs{G}{v}$ in place of $\obs{G}{\{v\}}$.

        To account for possible PMU failure, we consider the \emph{fragile power domination process}, first presented in~\cite{BBFK2023}.
        Given a graph $G$ with initial PMU placement $S\subseteq V(G)$ and sensor failure probability $q$, we run the following process:
        \begin{enumerate}
            \item[0.]
                \textbf{(Sensor Failure Step)}: Start with $S^* = \emptyset$.
                For each $v \in S$, independently add $v$ to $S^*$ with probability $(1-q)$.
                If $v$ is added to $S^*$, we say the sensor $v$ \emph{succeeds}.
                Otherwise, the sensor $v$ \emph{fails}.
            \item
                \textbf{(Domination Step)}: Let $B := N\left[S^*\right]$.
            \item 
                \textbf{(Zero Forcing Step)}: While there exists some vertex $x \in B$ such that there is a single vertex, say $y$, in $N[x] \setminus B$, add $y$ to $B$.
        \end{enumerate}
        The set $B$ at the end of applying the fragile power domination process is denoted $\obsq{G}{S}$.
        Note that $\obsq{G}{S}$ is a random variable.

        We are interested in how many vertices we expect to observe after running the fragile power domination process.
        In particular, we will concern ourselves with The \emph{expected value polynomial}, $\expolq{G}{S}: = \mathbb{E}\left[|\obsq{G}{S}|\right]$.
        Given a graph $G$, PMU placement set $S\subseteq V(G)$, and probability of PMU failure $q$,  observe that \begin{equation}\label{equation set view of E}
            \expolq{G}{S}=\sum_{W\subseteq S}|\obs{G}{W}|q^{|S\setminus W|}(1-q)^{|W|}. 
        \end{equation}
    
    \subsection{Overview of main results}

        In Section~\ref{sec:gadgets}, we develop structures that can be added to a graph that give control over the expected value polynomial for a given PMU placement. Our first result implies that we can add structure to a graph to control almost all the terms in the expected value polynomial.
        
        \begin{customthm}{\ref{theorem induced any polynomial}}
            Let $d\in\mathbb{N}$. Let $G$ be a graph and $S\subseteq V(G)$ with $|S|=d$.
            For any integers $c_d,c_{d-1},c_{d-2},\dots,c_2$, there exists a graph $G'$ with $G$ as an induced subgraph, and a choice of $c_1$ and $c_0$ such that
            \begin{align*}
                \expolq{G'}{S}=\sum_{k=0}^d c_kq^k.
            \end{align*}
        \end{customthm}
        
        Theorem~\ref{theorem induced any polynomial} will follow from a stronger statement involving multisets presented in Section~\ref{sec:gadgets}.
        In addition, we will show that Theorem~\ref{theorem induced any polynomial}, along with the multiset analogue, are best possible in the sense that the constant and linear coefficients cannot be controlled in the same way as the other coefficients.

        Section~\ref{sec:linearity} investigates when two graphs can have the same expected value polynomial, and when the expected value polynomial is lower degree than one might expect from the number of sensors.
        We provide a characterization of when two graphs (along with two PMU placements of the same size) can have the same expected value polynomial.
        
        \begin{customthm}{\ref{theorem linear algebraic characterization of copolynimal graphs}}
            Let $G$ and $G'$ be graphs, and let $S\subseteq V(G)$ and $S'\subseteq V(G')$ be such that $|S|=|S'|=:s$. The polynomials
            \[
                \expolq{G}{S}=\expol{G'}{S'}{q}
            \]
            if and only if for each $1\leq k\leq s$, 
            \[
                \sum_{W\in\binom{S}{k}}|\obs{G}{W}|=\sum_{W'\in\binom{S'}{k}}|\obs{G'}{W'}|.
            \]
        \end{customthm}
        
       The seminal work on power domination (\cite{BBFK2023}) gave a  sufficient condition for when an expected value polynomial is linear. We show that this condition is not necessary via an example presented in Section~\ref{sec:linearity}, and give a slightly stronger condition which is both necessary and sufficient in the following theorem. 
        
        \begin{customthm}{\ref{theorem linear}}
            Given a graph $G$ and a set $S\subseteq V(G)$, $\expolq{G}{S}$ is linear if and only if for all $1\leq k \leq |S|$ we have
            \[
            \sum_{W\in \binom{S}{k}} \sizeobs{G}{W} = \binom{s-1}{k-1} \sum_{v\in S} \sizeobs{G}{v}.
            \]
        \end{customthm}
        We also provide a conjecture which would give a necessary and sufficient condition the expected value polynomial $\expolq{G}{S}$ to be degree at most $\ell<|S|$. We show one direction of this conjecture, and that for many situations where $\expolq{G}{S}$ is quadratic it is also necessary.
        
        Finally, in Section~\ref{sec:highdeg} we give a method for determining when one placement of PMUs is better than another. Using this, we explore a common heuristic in power domination---that higher degree vertices tend to be better locations for PMUs---and show that in fragile power domination this is not always true.

    \subsection{Definitions, notation and tools}
        Given a set $A$, we will use $2^A$ to denote the power set of $A$, i.e. the set of all subsets of $A$, and $\binom{A}{k}$ to denote the set of $k$-element subsets of $A$.
        Given a function $f$, if $A$ is a subset of the domain of $f$, we write $f(A)$ to denote the image of $A$, i.e. $f(A)=\{f(a):a\in A\}$.
               
        Given a set $S$ and a function $f:S\to \mathbb{N}$, we will let $M(S,f)$ denote the multiset with support $S$ where each $s\in S$ has multiplicity $f(s)$. We note here that if $M=M(S,f)$ is a multiset placement of PMUs, then analogous to~\eqref{equation set view of E}, we can write
        \begin{equation}\label{equation multiset form of expol}
        \mathcal{E}(G;M,q)=\sum_{S'\subseteq S}|\mathrm{Obs}(G;S')|q^{f(S\setminus S')}\prod_{s\in S'}(1-q^{f(s)}).
        \end{equation}
        
        Another useful form for the expected value polynomial is
        \begin{equation}\label{equation vertex view of E}
            \expolq{G}{S}=\sum_{v\in V(G)}\mathrm{Pr}(v\text{ is observed}),
        \end{equation}
        where we note that~\eqref{equation vertex view of E} holds regardless of if $S$ is a set or multiset. 

        Given a polynomial $p$ over the variable $q$ and a non-negative integer $j$, we will write $p[q^j]$ to denote the coefficient of $q^j$ in the expanded form of $p$.

        Given a graph $G$ with a vertex $v\in V(G)$, if we \emph{add a leaf to $G$ at $v$}, we mean to add a single vertex $u\not\in V(G)$ to the graph, and then add the edge $uv$.
        Given a graph $G$, a set $F\subseteq V(G)$ is a \emph{fort} if there are no vertices in $V(G)\setminus F$ which have exactly one neighbor in $F$.
        Forts were first introduced by Fast and Hicks~\cite{FH2018} in the context of zero forcing, as a fort cannot be zero forced into.
        The \emph{entrance} of a fort $F$ is $N[F]\setminus F$, i.e. the collection of vertices which are not in $F$, but are adjacent to a vertex in $F$.
        In~\cite{BBEFFH2019}, Bozeman et al. showed that in power domination, a sensor must be placed either in $F$ or in the entrance of $F$ for vertices in $F$ to be observed.

\section{Controlling the Coefficients of the Expected Value Polynomial}\label{sec:gadgets}

    In this section, we prove a stronger version of Theorem~\ref{theorem induced any polynomial} for multisets, as well as show that these theorems are in some sense ``best possible''. In order to do this, we will first describe \emph{gadgets} which will help us modify the expected value polynomial of a given graph and PMU placement.

    \subsection{Gadgets}
        We will introduce gadgets that will be used to construct a graph $G'$ which has the graph $G$ as an induced subgraph.
        The structure of the gadgets will allow us to control the coefficients of $q^k$ for some values of $k$.
        One gadget will allow us to make the coefficient of $q^k$ arbitrarily large (i.e. positive), while the other will allow us to make the coefficient arbitrarily small (i.e. negative), with the direction of the gadget impact dependent upon the parity of $k$. 
        Each gadget will contain labeled vertices: a set of vertices called \emph{affix} vertices, one vertex called the \emph{connection vertex} and one vertex called the \emph{path head}.
        When we add a gadget to a graph $G$ with an initial PMU placement $S$, the affix vertices will be identified with some subset of $S$ (where we affix the gadget to $G$) and the path head vertex will lead to a path on the desired number of vertices for controlling the coefficient. 
        
        We now construct two gadgets for each $k\in \mathbb{N}$, $k\geq 2$.
        Examples of each are given in \cref{figure gadgets}.
    
        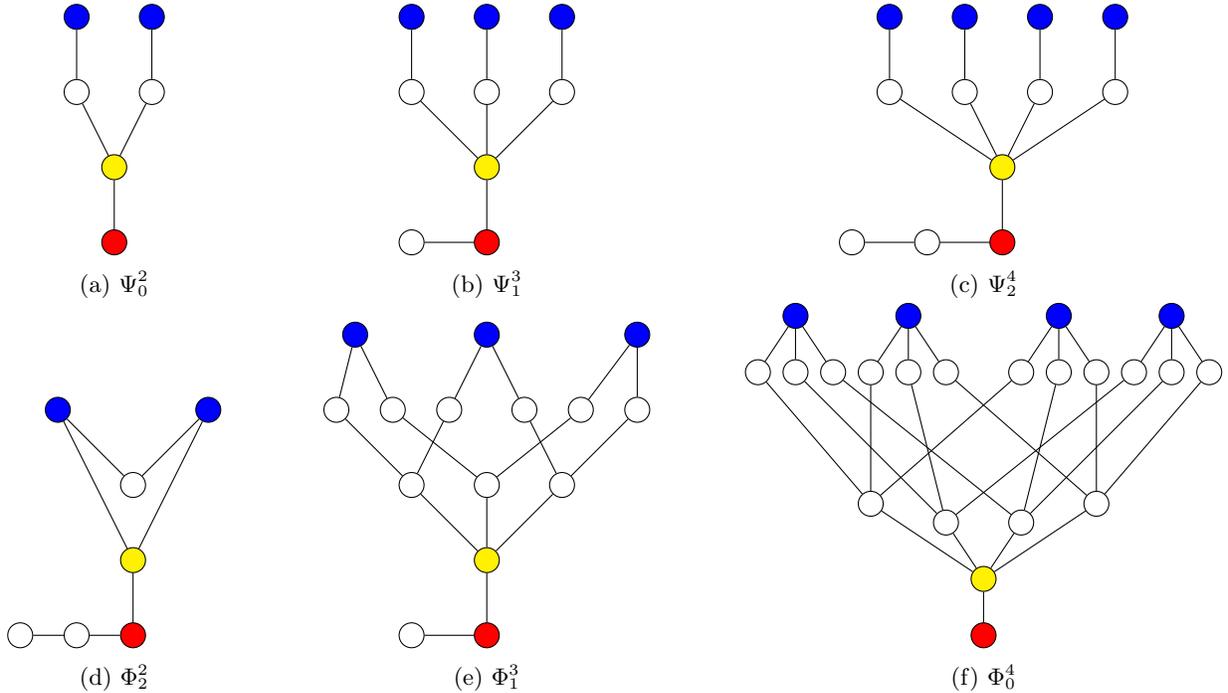
\begin{figure}[htb]
            \begin{subfigure}[t]{0.2\textwidth}
                \centering
                \begin{tikzpicture}[rotate=180]
                    \node [style=Yellow Node] (0) at (0, 1) {};
                    \node [style=Black Node] (1) at (-0.5, 0) {};
                    \node [style=Black Node] (2) at (0.5, 0) {};
                    \node [style=Red Node] (3) at (0, 2) {};
                    \node [style=Blue Node] (4) at (-0.5, -1) {};
                    \node [style=Blue Node] (5) at (0.5, -1) {};
                    \draw [style=Black Edge] (4) to (1);
                    \draw [style=Black Edge] (1) to (0);
                    \draw [style=Black Edge] (0) to (3);
                    \draw [style=Black Edge] (0) to (2);
                    \draw [style=Black Edge] (2) to (5);
                \end{tikzpicture}
                \caption{$\forka{0}{2}$}
            \end{subfigure}
            \hfill
            \begin{subfigure}[t]{0.3\textwidth}
                \centering
                \begin{tikzpicture}[rotate=180]
                    \node [style=Yellow Node] (0) at (0, 1) {};
                    \node [style=Black Node] (1) at (-1, 0) {};
                    \node [style=Black Node] (2) at (0, 0) {};
                    \node [style=Red Node] (3) at (0, 2) {};
                    \node [style=Black Node] (8) at (1, 2) {};
                    \node [style=Blue Node] (4) at (-1, -1) {};
                    \node [style=Blue Node] (5) at (0, -1) {};
                    \node [style=Blue Node] (6) at (1, -1) {};
                    \node [style=Black Node] (7) at (1, 0) {};
                    \draw [style=Black Edge] (4) to (1);
                    \draw [style=Black Edge] (1) to (0);
                    \draw [style=Black Edge] (0) to (3);
                    \draw [style=Black Edge] (0) to (2);
                    \draw [style=Black Edge] (2) to (5);
                    \draw [style=Black Edge] (0) to (7);
                    \draw [style=Black Edge] (7) to (6);
                    \draw [style=Black Edge] (3) to (8);
                \end{tikzpicture}
                \caption{$\forka{1}{3}$}
            \end{subfigure}
            \hfill
            \begin{subfigure}[t]{0.4\textwidth}
                \centering
                \begin{tikzpicture}[rotate=180]
                    \node [style=Yellow Node] (0) at (0, 1) {};
                    \node [style=Black Node] (1) at (-1.5, 0) {};
                    \node [style=Black Node] (2) at (-0.5, 0) {};
                    \node [style=Red Node] (3) at (0, 2) {};
                    \node [style=Black Node] (10) at (1, 2) {};
                    \node [style=Black Node] (11) at (2, 2) {};
                    \node [style=Blue Node] (4) at (-1.5, -1) {};
                    \node [style=Blue Node] (5) at (-0.5, -1) {};
                    \node [style=Blue Node] (6) at (0.5, -1) {};
                    \node [style=Black Node] (7) at (0.5, 0) {};
                    \node [style=Blue Node] (8) at (1.5, -1) {};
                    \node [style=Black Node] (9) at (1.5, 0) {};
                    \draw [style=Black Edge] (4) to (1);
                    \draw [style=Black Edge] (1) to (0);
                    \draw [style=Black Edge] (0) to (3);
                    \draw [style=Black Edge] (0) to (2);
                    \draw [style=Black Edge] (2) to (5);
                    \draw [style=Black Edge] (0) to (7);
                    \draw [style=Black Edge] (7) to (6);
                    \draw [style=Black Edge] (9) to (8);
                    \draw [style=Black Edge] (9) to (0);
                    \draw [style=Black Edge] (10) to (11);
                    \draw [style=Black Edge] (3) to (10);
                \end{tikzpicture}
                \caption{$\forka{2}{4}$}
            \end{subfigure}
            \begin{subfigure}[t]{0.2\textwidth}
                \centering
                \begin{tikzpicture}[rotate=180]
                    \node [style=Yellow Node] (0) at (0, 1) {};
                    \node [style=Red Node] (3) at (0, 2) {};
                    \node [style=Black Node] (14) at (0, 0) {};
                    \node [style=Blue Node] (18) at (-1, -1) {};
                    \node [style=Blue Node] (19) at (1, -1) {};
                    \draw [style=Black Edge] (0) to (3);
                    \draw [style=Black Edge] (18) to (14);
                    \draw [style=Black Edge] (19) to (14);
                    \draw [style=Black Edge] (0) to (18);
                    \draw [style=Black Edge] (0) to (19);
                    \node [style=Black Node] (22) at (.75, 2) {};
                    \node [style=Black Node] (23) at (1.5, 2) {};
                    \draw [style=Black Edge] (3) to (22);
                    \draw [style=Black Edge] (22) to (23);
                \end{tikzpicture}
                \caption{$\spoona{2}{2}$}
            \end{subfigure}
            \hfill
            \begin{subfigure}[t]{0.3\textwidth}
                \centering
                \begin{tikzpicture}[rotate=180]
                    \node [style=Yellow Node] (0) at (0, 1) {};
                    \node [style=Black Node] (1) at (-1, 0) {};
                    \node [style=Black Node] (2) at (0, 0) {};
                    \node [style=Red Node] (3) at (0, 2) {};
                    \node [style=Black Node] (21) at (1, 2) {};
                    \node [style=Black Node] (7) at (1, 0) {};
                    \node [style=Black Node] (10) at (-1.25, -1) {};
                    \node [style=Black Node] (11) at (0.5, -1) {};
                    \node [style=Black Node] (12) at (2, -1) {};
                    \node [style=Black Node] (14) at (-2, -1) {};
                    \node [style=Black Node] (15) at (-0.5, -1) {};
                    \node [style=Black Node] (16) at (1.25, -1) {};
                    \node [style=Blue Node] (18) at (-2, -2) {};
                    \node [style=Blue Node] (19) at (0, -2) {};
                    \node [style=Blue Node] (20) at (1.75, -2) {};
                    \draw [style=Black Edge] (1) to (0);
                    \draw [style=Black Edge] (0) to (3);
                    \draw [style=Black Edge] (0) to (2);
                    \draw [style=Black Edge] (0) to (7);
                    \draw [style=Black Edge] (14) to (1);
                    \draw [style=Black Edge] (10) to (2);
                    \draw [style=Black Edge] (15) to (1);
                    \draw [style=Black Edge] (11) to (7);
                    \draw [style=Black Edge] (16) to (2);
                    \draw [style=Black Edge] (12) to (20);
                    \draw [style=Black Edge] (20) to (16);
                    \draw [style=Black Edge] (11) to (19);
                    \draw [style=Black Edge] (19) to (15);
                    \draw [style=Black Edge] (10) to (18);
                    \draw [style=Black Edge] (18) to (14);
                    \draw [style=Black Edge] (12) to (7);
                    \draw [style=Black Edge] (21) to (3);
                \end{tikzpicture}
                \caption{$\spoona{1}{3}$}
            \end{subfigure}
            \hfill
            \begin{subfigure}[t]{0.4\textwidth}
                \centering
                \begin{tikzpicture}[rotate=180]
                    \node [style=Blue Node] (0) at (2.5, 0) {};
                    \node [style=Blue Node] (1) at (1, 0) {};
                    \node [style=Blue Node] (2) at (-1, 0) {};
                    \node [style=Blue Node] (3) at (-2.5, 0) {};
                    \node [style=Black Node] (4) at (-1.5, 2.5) {};
                    \node [style=Black Node] (5) at (-0.5, 2.75) {};
                    \node [style=Black Node] (6) at (0.5, 2.75) {};
                    \node [style=Black Node] (7) at (1.5, 2.5) {};
                    \node [style=Yellow Node] (8) at (0, 3.5) {};
                    \node [style=Red Node] (9) at (0, 4.25) {};
                    \node [style=Black Node] (10) at (2.5, 0.75) {};
                    \node [style=Black Node] (11) at (3, 0.75) {};
                    \node [style=Black Node] (12) at (2, 0.75) {};
                    \node [style=Black Node] (13) at (1, 0.75) {};
                    \node [style=Black Node] (14) at (1.5, 0.75) {};
                    \node [style=Black Node] (15) at (0.5, 0.75) {};
                    \node [style=Black Node] (16) at (-1, 0.75) {};
                    \node [style=Black Node] (17) at (-0.5, 0.75) {};
                    \node [style=Black Node] (18) at (-1.5, 0.75) {};
                    \node [style=Black Node] (19) at (-2.5, 0.75) {};
                    \node [style=Black Node] (20) at (-2, 0.75) {};
                    \node [style=Black Node] (21) at (-3, 0.75) {};
                    \draw [style=Black Edge] (9) to (8);
                    \draw [style=Black Edge] (8) to (5);
                    \draw [style=Black Edge] (8) to (4);
                    \draw [style=Black Edge] (8) to (6);
                    \draw [style=Black Edge] (8) to (7);
                    \draw [style=Black Edge] (0) to (11);
                    \draw [style=Black Edge] (10) to (0);
                    \draw [style=Black Edge] (12) to (0);
                    \draw [style=Black Edge] (15) to (1);
                    \draw [style=Black Edge] (13) to (1);
                    \draw [style=Black Edge] (14) to (1);
                    \draw [style=Black Edge] (18) to (2);
                    \draw [style=Black Edge] (16) to (2);
                    \draw [style=Black Edge] (17) to (2);
                    \draw [style=Black Edge] (21) to (3);
                    \draw [style=Black Edge] (19) to (3);
                    \draw [style=Black Edge] (20) to (3);
                    \draw [style=Black Edge] (20) to (6);
                    \draw [style=Black Edge] (19) to (5);
                    \draw [style=Black Edge] (21) to (4);
                    \draw [style=Black Edge] (17) to (7);
                    \draw [style=Black Edge] (16) to (5);
                    \draw [style=Black Edge] (18) to (4);
                    \draw [style=Black Edge] (14) to (7);
                    \draw [style=Black Edge] (13) to (6);
                    \draw [style=Black Edge] (15) to (4);
                    \draw [style=Black Edge] (11) to (7);
                    \draw [style=Black Edge] (6) to (10);
                    \draw [style=Black Edge] (12) to (5);
                \end{tikzpicture}
                \caption{$\spoona{0}{4}$}
            \end{subfigure}
            \caption{Both gadget types with 2, 3, and 4 affix vertices and varying appended path lengths. Blue vertices are affix vertices, red vertices are the path heads, and yellow vertices are the connection vertices.}\label{figure gadgets}
        \end{figure}
        \begin{construction}[Construction of a $\forka{\ell}{a}$-gadget for $a\geq 2$]\label{construction k1 gadget}
            This gadget consists of a star $K_{1,a+1}$ where $a$ of the edges are subdivided.
            The $a$ degree one vertices adjacent to degree two vertices will all be affix vertices, and the remaining degree one vertex will be the path head.
            The center of the star will be the connection vertex.
            A path with $\ell$ vertices is appended to the path head vertex.
        \end{construction}

        \begin{construction}[Construction of a $\spoona{\ell}{2}$-gadget]\label{construction 22 gadget}
            Start with a $C_4$ and add a leaf to any vertex. 
            The two vertices in the cycle which are adjacent to the degree $3$ vertex will be affix vertices. The degree $3$ vertex will be the connection vertex, and
            the degree $1$ vertex will be the path head.
            A path with $\ell$ vertices is appended to the path head vertex.
        \end{construction}
        
        \begin{construction}[Construction of a $\spoona{\ell}{a}$-gadget for $a\geq 3$]\label{construction spoon}
            Start with a set of vertices $A=\{w_1,w_2,\dots,w_a\}$.
            Then let $B=\binom{A}{a-1}$ consist of the set of $(a-1)$-element subsets of $A$.
            Construct a bipartite graph with parts $A$ and $B$, where $w\in A$ is adjacent to $b\in B$ if $w\in b$.
            Subdivide each edge in this bipartite graph.
            Add a vertex $x$ so that $x$ is adjacent to all $b\in B$.
            Add vertex $v$ adjacent to $x$. The vertex $x$ will be the connection vertex, the set $A$ will be the collection of affix vertices, and $v$ will be the path head vertex.
            A path with $\ell$ vertices is appended to the path head vertex.
        \end{construction}

        When we wish to reference a gadget without specifying whether it is a $\spoon{}{}$ or $\fork{}{}$-type gadget, we will write $\gadgeta{\ell}{a}$.
        We can now use gadgets in order to define the graph operation that will be used to construct a graph $G'$ from a given graph $G$.
    
        \begin{construction}[Affix Gadget Operation, $G' = \glueg{A}{\gadgeta{\ell}{|A|}}$]\label{construction affix gadget operation}
            Given a graph $G$, $A\subseteq V(G)$, and a gadget $\gadgeta{\ell}{|A|}$ we construct the graph $G' = G\boxminus_A \gadgeta{\ell}{|A|}$ by identifying each $w\in A$ with a unique affix vertex of $\gadgeta{\ell}{|A|}$. We say that the gadget $\gadgeta{\ell}{|A|}$ has been \emph{affixed} to the graph $G$ to create the graph $G'$.
        \end{construction}
        With PMUs located at affix vertices, the only way for the path head of any gadget to become observed is via the zero forcing step.
        It follows that the probability of observing any vertex in the appended path is the same as the probability of observing the path head.
        We formalize the probability of the path head vertex becoming observed in the following lemma.
        
        \begin{lemma}\label{lemma probability of path head being observed}
            Fix $a\geq 2$.
            Let $G$ be a graph and $S\subseteq V(G)$ be a set of vertices with $|S|\geq a$.
            Let $f:S\to \mathbb{N}$, let $M=M(S,f)$ be a multiset and let $A\subseteq S$ be a set of size $a$ such that every vertex in $A$ is adjacent to two leaves, both of which are not in $S$.
            Let $\gadgeta{\ell}{a}$ be a gadget with path head $v$.
            If $M$ is an initial placement of sensors which fail with probability $q$, then in $G' := \glueg{A}{\gadgeta{\ell}{a}}$ we have that
            \[
                \mathrm{Pr}(v\text{ is observed})=
                \begin{cases}
                    \displaystyle\prod_{w\in A}\left(1-q^{f(w)}\right) & \text{ if }\gadgeta{\ell}{a}=\forka{\ell}{a}, a\geq 2\\
                        &\\
                    1-q^{f(A)} & \text{ if } \gadgeta{\ell}{a}=\spoona{\ell}{2}\\
                        &\\
                    \displaystyle\sum_{w\in A} \left(q^{f(w)}\hspace{-6pt}\displaystyle\prod_{u\in A\setminus \{w\}}\left(1-q^{f(u)}\right)\right) +\displaystyle\prod_{w\in A}\left(1-q^{f(w)}\right) & \text{ if }\gadgeta{\ell}{a}=\spoona{\ell}{a}, a\geq 3

                \end{cases}
            \]
            Moreover, for any vertex $v'$ in the path appended to the path head $v$, we have that $\mathrm{Pr}(v'\text{ is observed}) = \mathrm{Pr}(v\text{ is observed})$.
        \end{lemma}
        
        \begin{proof}
            We note that in $G'$, each vertex in $A$ is adjacent to at least two leaves which are not in $S$, and so no vertices in $A$ could ever observe vertices via the zero forcing step alone.
            As such, the only way any vertex in $V(\gadgeta{\ell}{a})\setminus A$ can become observed is if sensors placed at vertices in $A$ avoid failure.
            We consider the following cases based on the type of gadget that $\gadgeta{\ell}{a}$ is.
    
            \noindent\textbf{Case 1:} $\gadgeta{\ell}{a}=\forka{\ell}{a}, a\geq 2$.
            
            For $v$ to be observed, it must be forced by the connection vertex and for this to happen, every vertex in $A$ must have at least one of its sensors from $M$ succeed, giving us the desired probability. 
            
            \noindent\textbf{Case 2:} $\gadgeta{\ell}{a}=\spoona{\ell}{2}$.
            
            If at least one sensor from $A$ succeeds, the entire gadget is observed, while if no sensors on $A$ succeed, $v$ is not observed.
            The probability that at least one vertex in $A$ has at least one sensor succeed is $1-q^{f(A)}$.
     
            \noindent\textbf{Case 3:} $\gadgeta{\ell}{a}=\spoona{\ell}{a}, a\geq 3$.
            
            Let $x$ be the connection vertex in $\spoona{\ell}{a}$.
            We claim that $x$ is observed if and only if at least $a-1$ vertices in $A$ have at least one sensor from $M$ succeed.
            Let $B$ be as in Construction~\ref{construction spoon}.
            Note that since every vertex in $B$ has degree $a$, the only way any of these vertices could perform a zero forcing step is if at least $a-1$ vertices in $A$ succeed, so $v$ is not observed if this does not happen.
            On the other hand, if at least $a-1$ vertices in $A$ succeed, say $w_1,\dots,w_{a-1}\in A$ all succeed, then the vertex $\{w_1,\dots,w_{a-1}\}\in B$ is able to force $x$.
            Additionally, every vertex in $B$ is observed since any $(a-1)$-set of $A$ has non-empty intersection with $\{w_1,\dots,w_{a-1}\}$, and so $x$ will force $v$ in this case, completing the proof of our claim.
        
            Thus, the probability that $v$ is observed is the probability that at least $a-1$ vertices in $A$ succeed.
            The expression $\displaystyle\sum_{w\in A} \left(q^{f(w)}\prod_{u\in A\setminus \{w\}}\left(1-q^{f(u)}\right)\right)$ gives the probability that exactly $a-1$ vertices in $A$ succeed. We could also have all $a$ vertices succeed, which occurs with probability  $\displaystyle\prod_{w\in A}\left(1-q^{f(w)}\right)$.
            These events are mutually exclusive and exhaustive, giving us the claimed probability.
    
            Now that we have established the probability that $v$ is observed, we note that the only way $v$ becomes observed is if the connection vertex forces $v$.
            In this case, the zero forcing process proceeds along the appended path until the entire appended path is observed, giving the moreover statement.
        \end{proof}
    
    \subsection{Controlling Coefficients}
        Given a graph $G$ with sensors at $S\subseteq V(G)$, we cannot hope to completely control every coefficient of $\expolq{G}{S}$.
        For example, we know that $\expolq{G}{S}[q^0]=|\obs{G}{S}|>0$, so we could not add structure to $G$ to cause the constant coefficient to become negative. Thus, we need to specify exactly which coefficients we can control.
        Towards this end, we define the set $\mathfrak{R}_2$ below, which we will show is the set of powers $k$ for which we can fully control the coefficient of $q^k$.
         
         Let $G$ be a graph, $S\subseteq V(G)$ and $f:S\to \mathbb{N}$. Let $M=M(S,f)$ be a multiset. Let 
         \[
            \mathfrak{R}_2:=\{f(S') \setgivensymbol S'\subseteq S, |S'|\geq 2\}.
         \] 
         That is, $\mathfrak{R}_2$ will consist of the set of multiplicities of sub\emph{sets} of at least 2 vertices for an initial PMU multiset placement $S$.
         For instance, with an initial PMU multiset placement with $S = \{a,b,c\}$ and $f: a\to 4,\, b \to 1, \, c\to 2$, we have $f(\{a,b\})=5$, $f(\{a,c\})=6$, $f(\{b,c\})=3$, and $f(\{a,b,c\})=7$. Thus gives us that $\mathfrak{R}_2=\{3,5,6,7\}$.
             
        We are now able to prove the following multiset generalization of Theorem~\ref{theorem induced any polynomial}, which shows that we can modify $G$ using gadgets to control the coefficient of $q^k$ for all $k\in \mathfrak{R}_2$.
        \begin{theorem}\label{theorem multiset induced any polynomial}
            Let $G$ be a graph, $S\subseteq V(G)$ and $f:S\to \mathbb{N}$.
            Let $M=M(S,f)$ be a multiset.
            Let $\mathfrak{R}_2:=\{f(S') \setgivensymbol S'\subseteq S, |S'|\geq 2\}$.
            For any collection of integers $\{c_i \setgivensymbol i\in \mathfrak{R}_2\}$, there exists a graph $G'$ such that
            \begin{equation}\label{equation desired coeffciient}
                \expolq{G'}{M}[q^k]=c_k
            \end{equation}
            for all $k\in \mathfrak{R}_2$ and $G$ is an induced subgraph of $G'$.
        \end{theorem}
        
        \begin{proof}
            For each $i\in \mathfrak{R}_2$, let $A_i\subseteq S$ be a set with $|A_i|\geq 2$ and $f(A_i)=i$, and let $\phi_i$ and $\psi_i$ be non-negative integers which we will specify later.
            For each $i$, affix the gadgets $\forka{\psi_i}{|A_i|}$ and $\spoona{\phi_i}{|A_i|}$ at $A_i$ as in \cref{construction affix gadget operation} (see Figure~\ref{fig:protofigone} for an example), and let us define $P_i^{\forka{}{}}$ and $P_i^{\spoona{}{}}$ to be the set of vertices in the appended path in $\forka{\psi_i}{|A_i|}$ and $\spoona{\phi_i}{|A_i|}$ respectively.
            In addition, if necessary, let us add leaves adjacent to the vertices in $S$ until every vertex in $S$ is adjacent to two leaves.
            Denote the resulting graph after affixing all gadgets and leaves by $G'$.
            Finally, define $P:=\bigcup_{i\in\mathfrak{R}_2}(P_i^{\forka{}{}}\cup P_i^{\spoona{}{}})$.
            \begin{figure}
                \centering
                \begin{subfigure}{.2\textwidth}
                    \centering
                    \begin{tikzpicture}
                        \node [style=Blue Node] (0) at (0, 0) {};
                        \node [style=Meta Node] (14) at (0.35, 0) {$u$};
                        \node [style=Blue Node] (1) at (1, 1) {};
                        \node [style=Meta Node] (15) at (1.35, 1) {$w$};
                        \node [style=Blue Node] (2) at (1, -1) {};
                        \node [style=Meta Node] (16) at (1.35, -1) {$v$};
                        \node [style=Black Node] (3) at (0.5, 2) {};
                        \node [style=Black Node] (4) at (1.5, 2) {};
                        \node [style=Black Node] (5) at (0.5, -2) {};
                        \node [style=Black Node] (6) at (1.5, -2) {};
                        \draw [style=Black Edge] (4) to (1);
                        \draw [style=Black Edge] (1) to (3);
                        \draw [style=Black Edge] (1) to (0);
                        \draw [style=Black Edge] (0) to (2);
                        \draw [style=Black Edge] (2) to (1);
                        \draw [style=Black Edge] (5) to (2);
                        \draw [style=Black Edge] (2) to (6);
                    \end{tikzpicture}
                    \caption{A Graph $G$ and set $S:=\{u,v,w\}$ indicated in blue.}
                    \label{fig:protofigonea}
                \end{subfigure}
                \hfill
                \begin{subfigure}{.75\textwidth}
                    \centering
                    \begin{tikzpicture}
                        \node [style=Meta Node] (-1) at (0.85, 0) {$u$};
                        \node [style=Meta Node] (-2) at (1, -1.35) {$v$};
                        \node [style=Meta Node] (-3) at (1, 1.35) {$w$};
                        \node [style=Meta Node] (-4) at (-4.5, 2) {};
                        \node [style=Meta Node] (-5) at (-4.5, -1.5) {};
                        \node [style=Meta Node] (-7) at (7, 2) {};
                        \node [style=Meta Node] (-8) at (-4, 2.5) {$\ell_1$};
                        \node [style=Meta Node] (-9) at (-4, -1) {$\ell_3$};
                        \node [style=Meta Node] (-10) at (4.5, -0.5) {$\ell_4$};
                        \node [style=Meta Node] (-11) at (6.5, 2.5) {$\ell_2$};
                        \node [style=Blue Node] (0) at (0.5, 0) {};
                        \node [style=Black Node] (3) at (0.5, -2) {};
                        \node [style=Black Node] (4) at (1.5, -2) {};
                        \node [style=Black Node] (5) at (1.5, 2) {};
                        \node [style=Black Node] (6) at (0.5, 2) {};
                        \node [style=Blue Node] (7) at (1, 1) {};
                        \node [style=Blue Node] (8) at (1, -1) {};
                        \node [style=Black Node] (9) at (-1.5, 3) {};
                        \node [style=Black Node] (10) at (-1.5, 2) {};
                        \node [style=Black Node] (11) at (-1.5, 1) {};
                        \node [style=Black Node] (12) at (-2.5, 2) {};
                        \node [style=Black Node] (14) at (-1.5, -1) {};
                        \node [style=Black Node] (15) at (-1.5, -2) {};
                        \node [style=Black Node] (16) at (-2.5, -1.5) {};
                        \node [style=Red Node] (17) at (-3.5, 2) {};
                        \node [style=Red Node] (18) at (-3.5, -1.5) {};
                        \node [style=Black Node] (19) at (3, 3.5) {};
                        \node [style=Black Node] (20) at (3, 3) {};
                        \node [style=Black Node] (21) at (3, 2.25) {};
                        \node [style=Black Node] (22) at (3, 1.75) {};
                        \node [style=Black Node] (23) at (3, 1) {};
                        \node [style=Black Node] (24) at (3, 0.5) {};
                        \node [style=Black Node] (25) at (4, 3.25) {};
                        \node [style=Black Node] (26) at (4, 2) {};
                        \node [style=Black Node] (27) at (4, 0.75) {};
                        \node [style=Black Node] (28) at (5, 2) {};
                        \node [style=Red Node] (29) at (6, 2) {};
                        \node [style=Black Node] (30) at (3, -1) {};
                        \node [style=Black Node] (31) at (3, -2) {};
                        \node [style=Black Node] (33) at (0.25, -0.75) {};
                        \node [style=Black Node] (34) at (0.25, 0.75) {};
                        \draw [style=Black Edge] (6) to (7);
                        \draw [style=Black Edge] (7) to (0);
                        \draw [style=Black Edge] (0) to (8);
                        \draw [style=Black Edge] (8) to (3);
                        \draw [style=Black Edge] (4) to (8);
                        \draw [style=Black Edge] (8) to (7);
                        \draw [style=Black Edge] (7) to (5);
                        \draw [style=Black Edge] (9) to (12);
                        \draw [style=Black Edge] (12) to (10);
                        \draw [style=Black Edge] (12) to (11);
                        \draw [style=Black Edge] (0) to (10);
                        \draw [style=Black Edge] (9) to (7);
                        \draw [style=Black Edge] (16) to (14);
                        \draw [style=Black Edge] (16) to (15);
                        \draw [style=Black Edge] (14) to (7);
                        \draw [style=Black Edge] (18) to (16);
                        \draw [style=Black Edge] (12) to (17);
                        \draw [style=Black Edge] (7) to (19);
                        \draw [style=Black Edge] (19) to (25);
                        \draw [style=Black Edge] (7) to (20);
                        \draw [style=Black Edge] (20) to (26);
                        \draw [style=Black Edge] (0) to (21);
                        \draw [style=Black Edge] (21) to (25);
                        \draw [style=Black Edge] (0) to (22);
                        \draw [style=Black Edge] (22) to (27);
                        \draw [style=Black Edge] (23) to (26);
                        \draw [style=Black Edge] (27) to (24);
                        \draw [style=Black Edge] (28) to (27);
                        \draw [style=Black Edge] (26) to (28);
                        \draw [style=Black Edge] (28) to (25);
                        \draw [style=Black Edge] (29) to (28);
                        \draw [style=Black Edge] (7) to (30);
                        \draw [style=Black Edge] (31) to (7);
                        \draw [style=Black Edge] (15) to (8);
                        \draw [style=Black Edge] (8) to (11);
                        \draw [style=Black Edge] (8) to (23);
                        \draw [style=Black Edge] (8) to (24);
                        \draw [style=Black Edge] (8) to (30);
                        \draw [style=Black Edge] (31) to (8);
                        \draw [style=Black Edge] (33) to (0);
                        \draw [style=Black Edge] (0) to (34);
                        \draw [style=Dashed Edge] (-4) to (17);
                        \draw [style=Dashed Edge] (-5) to (18);
                        \draw [style=Dashed Edge] (-7) to (29);
                        \node [style=Meta Node] (-6) at (5, -1) {};
                        \node [style=Red Node] (32) at (4, -1) {};
                        \draw [style=Dashed Edge] (-6) to (32);
                        \draw [style=Black Edge] (32) to (30);
                    \end{tikzpicture}
                    \caption{The Graph $G'$ resulting from affixing $\forka{\ell_1}{\{u,v,w\}}$, $\forka{\ell_3}{\{v,w\}}$, $\spoona{\ell_2}{\{u,v,w\}}$, and $\spoona{\ell_3}{\{v,w\}}$, along with appending two leaves to $u$. The set $S:=\{u,v,w\}$ indicated in blue and the path heads of all gadgets are indicated in red.}
                \end{subfigure}
                \caption{Graphs $G$ with $\expol{G}{S}{q}=-3q^3-4q+7$ and $G'$ with $\expol{G'}{S}{q}=(5-\ell_1+2\ell_2)q^3+(-21+3\ell_1-3\ell_2+\ell_3-\ell_4)q^2+(-16-3\ell_1-2\ell_3)q+(32+\ell_1+\ell_2+\ell_3+\ell_4)$. If we wanted to make e.g. $\expol{G'}{S}{q}[q^3]=\expol{G'}{S}{q}[q^2]=0$, we could do so by setting $\ell_1=5, \ell_2=0, \ell_3=6$, and $\ell_4=0$. This choice of path lengths results in $\expolq{G'}{S}=43(1-q)$ and the graph given in Figure~\ref{fig:3_2converse}.}
                \label{fig:protofigone}
            \end{figure}
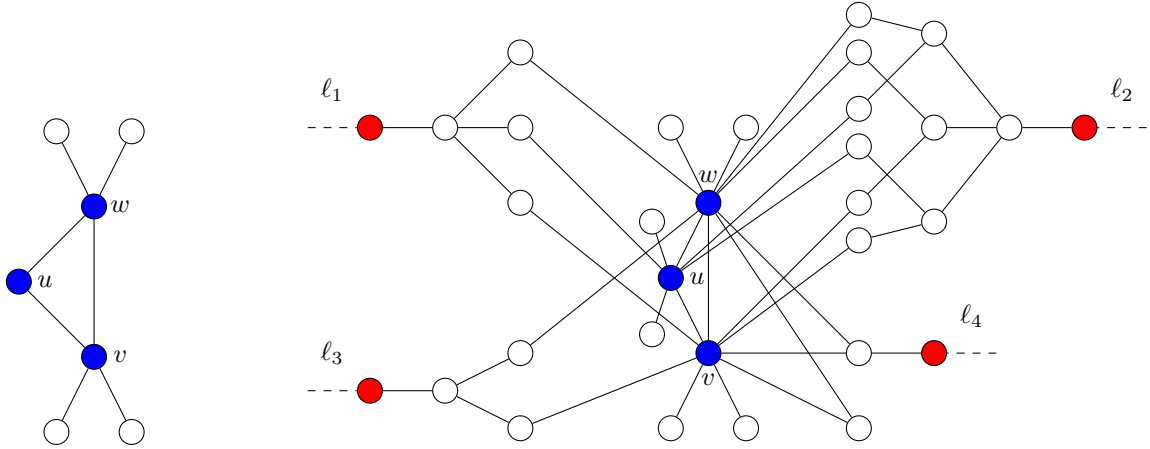
            
            Observe that
            \begin{equation}\label{equation gadgets proof vertex split}
                \expolq{G'}{S} = \sum_{v\in V(G')\setminus P} \Pr(v\text{ is observed}) + \sum_{v\in P}\Pr(v\text{ is observed}).
            \end{equation}
            Let us write $\displaystyle g(q):=\sum_{v\in V(G')\setminus P} \Pr(v\text{ is observed})$ and $ \displaystyle h_i(q):=\sum_{v\in P_i^{\spoona{}{}}\cup P_i^{\forka{}{}}}\Pr(v\text{ is observed})$, so that~\eqref{equation gadgets proof vertex split} becomes
            \begin{equation}\label{equation expressing the expected value polynomial in terms of path polynomials}
                \expolq{G'}{S}=g(q)+\sum_{i\in\mathfrak{R}_2} h_i(q).
            \end{equation}
            Let us explore the sum in~\eqref{equation expressing the expected value polynomial in terms of path polynomials} more carefully.
            For all $i\in\mathfrak{R}_2$, we claim that $h_i(q)$ is a polynomial in $q$ of degree at most $i$, and that for any integer $t$, there exists a choice of $\phi_i$ and $\psi_i$ such that $h_i(q)[q^i]=t$.
            Indeed, let $i\in\mathfrak{R}_2$, and let us consider cases based on the size of $A_i$.
        
            \noindent\textbf{Case 1:} $|A_i|=2$.
            Then by \cref{lemma probability of path head being observed} we have
            \begin{align*}
                h_i(q)&=\sum_{v\in P_i^{\forka{}{}}}\Pr(v\text{ is observed}) + \sum_{v\in P_i^{\spoona{}{}}}\Pr(v\text{ is observed})\\
                &= \psi_2 \left( \displaystyle\prod_{w\in A_i}\left(1-q^{f(w)}\right) \right) +\phi_2 \left(1-q^{f(A_i)} \right).
            \end{align*}
            The leading term for the left is $\psi_2q^{f(A_i)}$ and the leading term for the right is $-\phi_2q^{f(A_i)}$.
            Since we have $f(A_i)=i$ by our choice of $A_i$, $\mathrm{deg}(h_i)\leq i$.
            Furthermore, we have that
            \[
                t_i:=h_i(q)[q^i]=\psi_2-\phi_2.
            \]
            Note that as the difference of two nonnegative integers, we can select $\phi_2$ and $\psi_2$ so that $t_i$ is any integer.
            
            \noindent\textbf{Case 2:} $|A_i|\geq 3$.
            By \cref{lemma probability of path head being observed} we have
            \begin{align*}
                h_i(q)&=\sum_{v\in P_i^{\forka{}{}}}\Pr(v\text{ is observed}) + \sum_{v\in P_i^{\spoona{}{}}}\Pr(v\text{ is observed})\\
                &= \psi_i\left( \prod_{w\in A_i}\left(1-q^{f(w)}\right)\right) + \phi_i \left(\sum_{w\in A_i} \left(q^{f(w)}\hspace{-6pt}\displaystyle\prod_{u\in A_i\setminus \{w\}}\left(1-q^{f(u)}\right)\right) +\displaystyle\prod_{w\in A_i}\left(1-q^{f(w)}\right) \right).
            \end{align*}
            The leading term for the left is $\psi_i (-1)^{|A_i|} q^i$ and for the right we have $\phi_i (-1)^{|A_i|-1} (|A_i|-1)q^{i}$.
            Thus, we again find that $h_i(q)$ has degree at most $i$. Furthermore, the coefficient of $q^i$ is 
            \begin{align*}
                t_i:= h_i(q)[q^i]&=(-1)^{|A_i|}\psi_i + (-1)^{|A_i|-1}\phi_i(|A_i|-1)\\
                &=(-1)^{|A_i|}(\psi_i-\phi_i(|A_i|-1)).
            \end{align*}
            Thus, we can choose $\phi_i$ and $\psi_i$ so that $t_i$ is any integer. This completes the proof of our claim.

            Now, from~\eqref{equation expressing the expected value polynomial in terms of path polynomials}, we can see that for all $i\in\mathfrak{R}_2$,
                \begin{align}
                \expolq{G'}{S}[q^i]&=g[q^i]+\sum_{i'\in\mathfrak{R}_2,i'\geq i} h_{i'}[q^i]\nonumber\\
                &=g[q^i]+\sum_{i'\in\mathfrak{R}_2,i'> i} h_{i'}[q^i]+t_i\label{equation coefficients of expol}
            \end{align}
        
            We can now choose $\phi_i$ and $\psi_i$ for each $i\in\mathfrak{R}_2$.
            Let $\alpha=f(M)$ be the largest element of $\mathfrak{R}_2$.
            Then from~\eqref{equation coefficients of expol}, we have $\expolq{G'}{S}[q^\alpha] = g[q^\alpha] + t_\alpha$.
            Thus, we choose $\phi_\alpha$ and $\psi_\alpha$ so that $t_\alpha = c_\alpha - g[q^\alpha]$, and consequentially $\expolq{G}{S}[q^\alpha] = c_\alpha.$
        
            We then move to the next smallest element of $\mathfrak{R}_2$, say $\beta$.
            Then from~\eqref{equation coefficients of expol}, we have $\expolq{G'}{S}[q^\beta] = g[q^\beta] + h_{\alpha}[q^{\beta}]+ t_\beta$.
            Then we can choose $\phi_\beta$ and $\psi_\beta$ so that $t_\beta = c_\beta - g[q^\beta] -h_{\alpha}[q^{\beta}]$, giving us $\expolq{G'}{S}[q^\beta] = c_\beta$.
        
            Continuing in this fashion, let $i\in\mathfrak{R}_2$ us assume we have chosen $\phi_{i'}$ and $\psi_{i'}$ for all $i'\in\mathfrak{R}_2$, $i'>i$.
            With~\eqref{equation coefficients of expol} in mind, we choose $\phi_i$ and $\psi_i$ such that 
            \[
                t_i=c_i-g[q^i]-\sum_{i'\in\mathfrak{R}_2,i'> i} h_{i'}[q^i].
            \]
            We note here in particular that our choice of $\phi_i$ and $\psi_i$ does not depend on the value of any $\phi_j$ or $\psi_j$ with $j<i$, only on those values which we have already chosen.
        
            We may continue in this fashion, choosing $\phi_i$ and $\psi_i$ until we have $\expolq{G'}{S}[q^i]=c_i$ for all $i\in\mathfrak{R}_2$, as desired.
        
            Moreover, $G$ is an induced subgraph of $G'$, as the construction does not add any edges between the original vertices of $G$.            
        \end{proof}

        We can now prove \cref{theorem induced any polynomial} by restricting from multiset PMU placements to set PMU placements.
        \begin{theorem}\label{theorem induced any polynomial}
            Let $d\in\mathbb{N}$.
            Let $G$ be a graph and $S\subseteq V(G)$ with $|S|=d$.
            For any integers $c_d,c_{d-1},c_{d-2},\dots,c_2$, there exists a graph $G'$ with $G$ as an induced subgraph along with values $c_1$ and $c_0$ such that
            \begin{equation}\label{equation induced any polynomial}
                \expolq{G'}{S}=\sum_{k=0}^d c_kq^k.
            \end{equation}
        \end{theorem}
    
        \begin{proof}
            We apply Theorem~\ref{theorem multiset induced any polynomial} to $G$ and $S$, with $f(v)=1$ for all $v\in S$, noting that $\mathfrak{R}_2=\{k\in \mathbb{N} \setgivensymbol 1 < k \leq |S|\}$.
            The graph $G'$ satisfies the desired conditions.
        \end{proof}
    
        Theorems~\ref{theorem induced any polynomial} and \ref{theorem multiset induced any polynomial} only claim we can control the coefficients of powers that appear in $\mathfrak{R}_2$.
        Our next result shows that while we do not have the same control for the remaining coefficients, we do have some knowledge of their form. Note that when $S$ is a set, the only coefficients not corresponding to elements of $\mathfrak{R_2}$ are the linear and constant coefficients.
        \begin{proposition}\label{proposition uncontrollable conefficients}
            Let $G$ be a graph, $S\subseteq V(G)$ and $f:S\to \mathbb{N}$, and let $d:=\sum_{s\in S}f(s)$, i.e. the total number of PMUs.
            Let $M=M(S,f)$ be a multiset.
            Let $\mathfrak{R}_1=\{f(S') \setgivensymbol S'\subseteq S, |S'|\geq 1\}$, and let $\mathfrak{R}_2=\{f(S') \setgivensymbol S'\subseteq S, |S'|\geq 2\}$.
            \begin{enumerate}
                \item
                    If $k\in [d]\setminus \mathfrak{R}_1$, then $\expolq{G}{M}[q^k]=0$.\label{item coefficient missing from polynomial}
                \item
                    If $k\in \mathfrak{R}_1\setminus \mathfrak{R}_2$, then $\expolq{G}{M}[q^k]\leq 0$.\label{item cannot control coefficient completely}
                \item
                    If $k=0$, then $\expolq{G}{M}[q^k]=\sizeobs{G}{S}\geq 0$.\label{item constant coefficient}
            \end{enumerate}
        \end{proposition}
    
        \begin{proof}
            Recalling~\eqref{equation multiset form of expol}, we have
            \begin{equation}\label{equation controlling coefficients multiset expected value polynomial}
                \expolq{G}{M}=\sum_{S'\subseteq S}\sizeobs{G}{S'}q^{f(S\setminus S')}\prod_{s\in S'}(1-q^{f(s)}).
            \end{equation}
            For~\ref{item coefficient missing from polynomial}, let $k\in [d] \setminus \mathfrak{R}_1$.
            We note that the only powers of $q$ that come from the expanded form of~\eqref{equation controlling coefficients multiset expected value polynomial} are powers which are sums of outputs of $f$.
            Since $k\not\in \mathfrak{R}_1$, this does not appear as an exponent.
            
            For~\ref{item cannot control coefficient completely}, let $k\in \mathfrak{R}_1\setminus \mathfrak{R}_2$, let $S^*=\{s \in S : f(s)=k\}$ so that $|S^*|=\ell \geq 1$.
            Since $k\not\in \mathfrak{R}_2$ and in the expanded form of~\eqref{equation controlling coefficients multiset expected value polynomial}, each power of $q$ is a sum of outputs of $f$, the only way to get a contribution to the coefficient of $q^k$ is if the sum of the outputs of $f$ is a $1$-term sum, corresponding to $s\in S$ with $f(s)=k$.
            Such contributions come from two types of terms from~\eqref{equation controlling coefficients multiset expected value polynomial},
            \begin{itemize}
                \item
                    The term with $S'=S$, $\displaystyle\sizeobs{G}{S}\prod_{s\in S}(1-q^{f(s)})$, when fully expanded contains $\ell$ terms with $q^k$, each with a coefficient of $-\sizeobs{G}{S}$, giving us a total contribution of $-\ell \sizeobs{G}{S}$.
                \item
                    For each $v\in S^*$, the term with $S'=S\setminus \{v\}$, $\displaystyle\sizeobs{G}{S\setminus\{v\}}q^{f(v)}\prod_{s\in S\setminus\{v\}}(1-q^{f(s)})$, when fully expanded contains a single term with $q^k$, with coefficient $\sizeobs{G}{S\setminus\{v\}}$. 
            \end{itemize}
            When adding these contributions up, we find the coefficient of $q^k$ to be
            \[
                \left(\sum_{v\in S^*}\sizeobs{G}{S\setminus\{v\}}\right)-\ell\sizeobs{G}{S}.
            \]
            Since the sum on the left has $\ell$ terms, and $\obs{G}{S\setminus\{v\}}\subseteq\obs{G}{S}$ for any $v\in S$, we have that the coefficient of $q^k$ is always non-positive.
            
            Finally, for~\ref{item constant coefficient}, we note that $\sizeobs{G}{S}=\expol{G}{M}{0}=\expolq{G}{M}[q^0]$.
        \end{proof}
        
        Proposition~\ref{proposition uncontrollable conefficients} shows that the coefficients which do not appear in Theorems~\ref{theorem induced any polynomial} or \ref{theorem multiset induced any polynomial} cannot be fully controlled by adding structure to a pre-existing graph.
        The coefficients corresponding to powers in $\mathfrak{R}_1$ can be made arbitrarily small (negative) by appending many leaves to the appropriate vertex in $S$.

\section{Characterizing the Expected Value Polynomial}\label{sec:linearity}

    In this section we will prove Theorem~\ref{theorem linear algebraic characterization of copolynimal graphs}, provide a counterexample to the converse of Theorem~\ref{theorem old linear} and prove Theorem~\ref{theorem linear}.
    Finally, we explore an extension of Theorem~\ref{theorem linear} to expected value polynomials of bounded degree higher than $1$.
    
    \begin{theorem}\label{theorem linear algebraic characterization of copolynimal graphs}
        Let $G$ and $G'$ be graphs, and let $S\subseteq V(G)$ and $S'\subseteq V(G')$ be such that $|S|=|S'|=:s$. The polynomials
        \[
            \expolq{G}{S}=\expol{G'}{S'}{q}
        \]
        if and only if for each $0\leq k\leq s$, 
        \[
            \sum_{W\in\binom{S}{k}}|\obs{G}{W}|=\sum_{W'\in\binom{S'}{k}}|\obs{G'}{W'}|.
        \]
    \end{theorem}
    
    \begin{proof}
        For $0\leq k\leq s$, define
        \[
            a_k:=\sum_{W\in \binom{S}{k}}\sizeobs{G}{W}.
        \]
        We can write
        \begin{align*}
            \expolq{G}{S}&=\sum_{W\subseteq S}\sizeobs{G}{W}q^{|S\setminus W|}(1-q)^{|W|}\\
            &=\sum_{k=0}^s\sum_{W\in \binom{S}{k}}\sizeobs{G}{W}q^{s-k}(1-q)^{k}\\
            &=\sum_{k=0}^sa_kq^{s-k}(1-q)^{k}.
        \end{align*}
        The set $\{q^{s-k}(1-q)^{k} \setgivensymbol 0\leq k\leq s\}$ is a basis\footnote{The basis provided here is similar to the so-called Bernstein basis (see e.g.\cite{Bernstein1912}), differing only by constant multiples.} for the vector space of polynomials of degree at most $d$ over $q$. If we similarly define 
        \[
            a_k':=\sum_{W'\in \binom{S'}{k}}\sizeobs{G'}{W'},
        \]
        then the equation $\expolq{G}{S}=\expolq{G'}{S'}$
        is equivalent to $a_k=a_k'$ for all $0\leq k\leq s$.
    \end{proof}

    We next give an example which shows that the hypothesis of of the linearity condition given in \cite{BBFK2023} is not sufficient to guarantee a linear expected value polynomial. 
    
    \begin{theorem}[\cite{BBFK2023}, Proposition~3.2]\label{theorem old linear}
            Let $G$ be a graph, $S\subseteq V(G)$ be a set, and $q$ a probability of PMU failure. If for every $X\subseteq S$,
            \begin{equation}
                |\obs{G}{X}|=\sum_{v\in X}|\obs{G}{v}|, \label{equation old linear}
            \end{equation}
            then $\expolq{G}{S}=|\obs{G}{S}|(1-q)$.
    \end{theorem}    
    Consider the graph $G$ in Figure~\ref{fig:3_2converse} with sensors placed at vertices $u$, $v$ and $w$.
    Note that we constructed this example using the tools developed in Section~\ref{sec:gadgets}.
    In Table~\ref{table table}, we give the size of $\obs{G}{X}$ for each $X \subseteq \{u,v,w\}$.
    Particularly, $\sizeobs{G}{\{u,v\}}=26$ but $\sizeobs{G}{u}+\sizeobs{G}{v}=11+16\neq26$, and so~\eqref{equation old linear} does not hold for every $X\subseteq \{u,v,w\}$.
    We further use values from Table~\ref{table table} and~\eqref{equation set view of E} to calculate that
    \[
    \expolq{G}{\{u,v,w\}}=(11+16+16)(1-q)q^2+(26+26+34)(1-q)^2q+43(1-q)^3=43(1-q).
    \]
    This shows that the converse of Theorem~\ref{theorem old linear} does not hold in general, and so the stronger condition we include in Theorem~\ref{theorem linear} is needed. 
    
    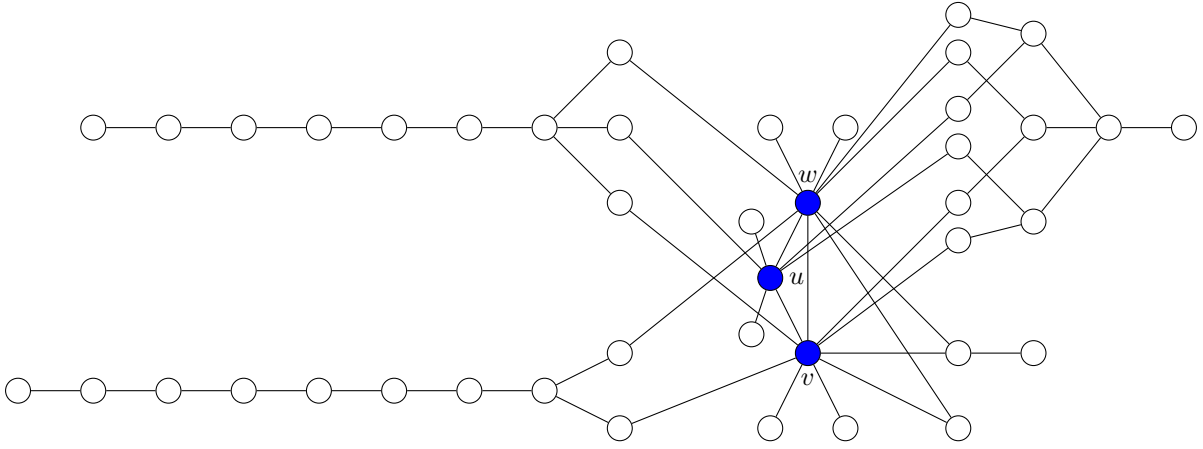
\begin{figure}
        \centering
        \begin{tikzpicture}
            \node [style=Meta Node] (-1) at (0.85, 0) {$u$};
            \node [style=Meta Node] (-2) at (1, -1.35) {$v$};
            \node [style=Meta Node] (-3) at (1, 1.35) {$w$};
            \node [style=Black Node] (32) at (4, -1) {};
            \node [style=Blue Node] (0) at (0.5, 0) {};
            \node [style=Black Node] (3) at (0.5, -2) {};
            \node [style=Black Node] (4) at (1.5, -2) {};
            \node [style=Black Node] (5) at (1.5, 2) {};
            \node [style=Black Node] (6) at (0.5, 2) {};
            \node [style=Blue Node] (7) at (1, 1) {};
            \node [style=Blue Node] (8) at (1, -1) {};
            \node [style=Black Node] (9) at (-1.5, 3) {};
            \node [style=Black Node] (10) at (-1.5, 2) {};
            \node [style=Black Node] (11) at (-1.5, 1) {};
            \node [style=Black Node] (12) at (-2.5, 2) {};
            \node [style=Black Node] (14) at (-1.5, -1) {};
            \node [style=Black Node] (15) at (-1.5, -2) {};
            \node [style=Black Node] (16) at (-2.5, -1.5) {};
            \node [style=Black Node] (17) at (-3.5, 2) {};
            \node [style=Black Node] (-10) at (-4.5, 2) {};
            \node [style=Black Node] (-11) at (-5.5, 2) {};
            \node [style=Black Node] (-12) at (-6.5, 2) {};
            \node [style=Black Node] (-13) at (-7.5, 2) {};
            \node [style=Black Node] (-14) at (-8.5, 2) {};
            \node [style=Black Node] (18) at (-3.5, -1.5) {};
            \node [style=Black Node] (-20) at (-4.5, -1.5) {};
            \node [style=Black Node] (-21) at (-5.5, -1.5) {};
            \node [style=Black Node] (-22) at (-6.5, -1.5) {};
            \node [style=Black Node] (-23) at (-7.5, -1.5) {};
            \node [style=Black Node] (-24) at (-8.5, -1.5) {};
            \node [style=Black Node] (-25) at (-9.5, -1.5) {};
            \node [style=Black Node] (19) at (3, 3.5) {};
            \node [style=Black Node] (20) at (3, 3) {};
            \node [style=Black Node] (21) at (3, 2.25) {};
            \node [style=Black Node] (22) at (3, 1.75) {};
            \node [style=Black Node] (23) at (3, 1) {};
            \node [style=Black Node] (24) at (3, 0.5) {};
            \node [style=Black Node] (25) at (4, 3.25) {};
            \node [style=Black Node] (26) at (4, 2) {};
            \node [style=Black Node] (27) at (4, 0.75) {};
            \node [style=Black Node] (28) at (5, 2) {};
            \node [style=Black Node] (29) at (6, 2) {};
            \node [style=Black Node] (30) at (3, -1) {};
            \node [style=Black Node] (31) at (3, -2) {};
            \node [style=Black Node] (33) at (0.25, -0.75) {};
            \node [style=Black Node] (34) at (0.25, 0.75) {};
            \draw [style=Black Edge] (6) to (7);
            \draw [style=Black Edge] (7) to (0);
            \draw [style=Black Edge] (0) to (8);
            \draw [style=Black Edge] (8) to (3);
            \draw [style=Black Edge] (4) to (8);
            \draw [style=Black Edge] (8) to (7);
            \draw [style=Black Edge] (7) to (5);
            \draw [style=Black Edge] (9) to (12);
            \draw [style=Black Edge] (12) to (10);
            \draw [style=Black Edge] (12) to (11);
            \draw [style=Black Edge] (0) to (10);
            \draw [style=Black Edge] (9) to (7);
            \draw [style=Black Edge] (16) to (14);
            \draw [style=Black Edge] (16) to (15);
            \draw [style=Black Edge] (14) to (7);
            \draw [style=Black Edge] (18) to (16);
            \draw [style=Black Edge] (12) to (17);
            \draw [style=Black Edge] (7) to (19);
            \draw [style=Black Edge] (19) to (25);
            \draw [style=Black Edge] (7) to (20);
            \draw [style=Black Edge] (20) to (26);
            \draw [style=Black Edge] (0) to (21);
            \draw [style=Black Edge] (21) to (25);
            \draw [style=Black Edge] (0) to (22);
            \draw [style=Black Edge] (22) to (27);
            \draw [style=Black Edge] (23) to (26);
            \draw [style=Black Edge] (27) to (24);
            \draw [style=Black Edge] (28) to (27);
            \draw [style=Black Edge] (26) to (28);
            \draw [style=Black Edge] (28) to (25);
            \draw [style=Black Edge] (29) to (28);
            \draw [style=Black Edge] (7) to (30);
            \draw [style=Black Edge] (31) to (7);
            \draw [style=Black Edge] (15) to (8);
            \draw [style=Black Edge] (8) to (11);
            \draw [style=Black Edge] (8) to (23);
            \draw [style=Black Edge] (8) to (24);
            \draw [style=Black Edge] (8) to (30);
            \draw [style=Black Edge] (31) to (8);
            \draw [style=Black Edge] (33) to (0);
            \draw [style=Black Edge] (0) to (34);
            \draw [style=Black Edge] (32) to (30);
            \draw [style=Black Edge] (17) to (-10);
            \draw [style=Black Edge] (-11) to (-10);
            \draw [style=Black Edge] (-12) to (-11);
            \draw [style=Black Edge] (-13) to (-12);
            \draw [style=Black Edge] (-14) to (-13);
            \draw [style=Black Edge] (18) to (-20);
            \draw [style=Black Edge] (-21) to (-20);
            \draw [style=Black Edge] (-21) to (-22);
            \draw [style=Black Edge] (-22) to (-23);
            \draw [style=Black Edge] (-23) to (-24);
            \draw [style=Black Edge] (-24) to (-25);
        \end{tikzpicture}
        \caption{A graph $G$ where $\expolq{G}{\{u,v,w\}}$ is linear, but~\eqref{equation old linear} is not satisfied for all $X\subseteq \{u,v,w\}$, showing that the converse of Theorem~\ref{theorem old linear} is false. See Table~\ref{table table} for $\sizeobs{G}{X}$ for each $X$.}
        \label{fig:3_2converse}
    \end{figure}
    
    \begin{table}
    \centering
        \begin{tabular}{|c|c|}
            \hline
            Subset&Number of observed vertices\\
            \hline
            \{u\}&11\\
            \{v\}&16\\
            \{w\}&16\\
            \hline
            \{u,v\}&26\\
            \{u,w\}&26\\
            \{v,w\}&34\\
            \hline
            \{u,v,w\}&43\\
            \hline
        \end{tabular}
        \caption{The number of observed vertices for each subset of $\{u,v,w\}$ from Figure~\ref{fig:3_2converse}.}\label{table table}
    \end{table}
    
    \begin{theorem}\label{theorem linear}
        Given a graph $G$ and a set $S\subseteq V(G)$ with $|S|=s$, $\expolq{G}{S}$ is linear if and only if for all $1\leq k \leq s$ we have
        \begin{equation}\label{lineariffcondition}
            \sum_{W\in \binom{S}{k}} \sizeobs{G}{W} = \binom{s-1}{k-1} \sum_{v\in S} \sizeobs{G}{v}.
        \end{equation}
    \end{theorem}
    \begin{proof}
        Let $s:=|S|$.
        If \eqref{lineariffcondition} is satisfied for all $1\leq k \leq s$, then observe that
        \begin{align*}
            \expolq{G}{S}   &= \sum_{W\subseteq S} \sizeobs{G}{W} q^{|S\setminus W|} (1-q)^{|W|} \\
                            &= \sum_{k=1} ^{s} \sum_{W\in \binom{S}{k}} \sizeobs{G}{W} q^{s-k} (1-q)^{k} \\
                            &= \sum_{k=1} ^{s}\binom{s-1}{k-1}  \sum_{v\in S} \sizeobs{G}{v} q^{s-k} (1-q)^{k} \\
                            &=(1-q) \sizeobs{G}{S} \sum_{k=1}^{s} \binom{s-1}{k-1} q^{s-k} (1-q)^{k-1}
        \end{align*}
        By the Binomial Theorem (see e.g. Theorem 5.2.1 in~\cite{B2010}), $\displaystyle\sum_{k=1}^{s} \binom{s-1}{k-1} q^{s-k} (1-q)^{k-1}=1$, so
        \[
            \expolq{G}{S} = (1-q) \sizeobs{G}{S}.
        \]
        Now let us assume that $\expolq{G}{S}$ is linear.
        Let $\ell:=\sizeobs{G}{S}$.
        First off, since $\expol{G}{S}{0}=\ell$ and $\expol{G}{S}{1}=0$, we have $\expolq{G}{S}=\ell(1-q)$.
        Let us consider the graph $G':=K_{\ell-s+1}\sqcup \overline{K_{s-1}}$ and let $S'\subseteq V(G')$ be a set of size $s$ that contains one vertex from each component of $G'$.
        We can compute
        \[
            \expolq{G'}{S'}=\sum_{v\in V(G')}\mathrm{Pr}(v\text{ is observed})=\ell(1-q)=\expolq{G}{S},
        \]
        since $G'$ has $\ell$ vertices, each of which is observed if and only if a single sensor does not fail.
        By Theorem~\ref{theorem linear algebraic characterization of copolynimal graphs}, for each $1\leq k\leq s$,
        \begin{equation}\label{equation sum of obs are equal}
            \sum_{W\in\binom{S}{k}}\sizeobs{G}{W}=\sum_{W'\in\binom{S'}{k}}\sizeobs{G'}{W'}.
        \end{equation}
        Furthermore, note that $G'$ has the property that for any $W'\subseteq S'$, 
        \begin{equation}\label{equation strong observation property}
            \sizeobs{G'}{W'}=\sum_{v\in W'}\sizeobs{G'}{v},
        \end{equation}
        since each vertex observes its entire component, and no two vertices of $S'$ are in the same component of $G'$.
        Then summing~\eqref{equation strong observation property} over all subsets $W'$ of a fixed size $k$, $1\leq k\leq s$ yields that
        \begin{equation}\label{equation obs in G' act like we expect}
            \sum_{W'\in \binom{S'}{k}}\sizeobs{G'}{W'}=\sum_{W'\in \binom{S'}{k}}\sum_{v\in W'}\sizeobs{G'}{v}=\binom{s-1}{k-1}\sum_{v\in S}\sizeobs{G'}{v}.
        \end{equation}
        Then, for any $1\leq k\leq s$, combining~\eqref{equation sum of obs are equal} and~\eqref{equation obs in G' act like we expect}, we have that
        \[
            \sum_{W\in \binom{S}{k}} \sizeobs{G}{W} = \sum_{W'\in \binom{S'}{k}}\sizeobs{G'}{W'}=\binom{s-1}{k-1}\sum_{v\in S'}\sizeobs{G'}{v}= \binom{s-1}{k-1}\sum_{v\in S}  \sizeobs{G}{v}.
        \]
    \end{proof}

    A natural question one may ask is if Theorem~\ref{theorem linear} can be generalized to understand when a graph $G$ with sensor placement $S$ will have quadratic or cubic expected value polynomial, or in general will have degree at most $\ell$ for some $\ell\leq |S|$.
    We provide the following conjecture.

    \begin{conjecture}\label{conjecture degree L}
        Given a graph $G$ and a set $S\subseteq V(G)$ with $|S|=s$, $\mathcal{E}(G,S,q)$ is degree at most $\ell$ if and only if for all $1\leq k\leq s$,
        \begin{equation}\label{equation condition for degree at most L}
            \sum_{W\in \binom{S}{k}}|\mathrm{Obs}(G,W)|=\sum_{i=1}^{\ell}\alpha_{s,\ell}(k,i)\sum_{W\in \binom{S}{i}}|\mathrm{Obs}(G,W)|,
        \end{equation}
        where
        \begin{equation}\label{equation definition of the alphas}
            \alpha_{s,\ell}(k,i):=\binom{s-i}{k-i}-\sum_{j=i+1}^\ell \alpha_{s,\ell}(k,j)\cdot \binom{s-i}{j-i}.
        \end{equation}
    \end{conjecture}

    Since the parameters $\alpha_{s,\ell}(k,i)$ are defined in reference to each other, it is not obvious that these parameters are well-defined.
    We address this concern in Lemma~\ref{lemma alphas are degree at most L}.
    We note that Conjecture~\ref{conjecture degree L} with $\ell=1$ corresponds exactly to Theorem~\ref{theorem linear}.
    We will show that the reverse direction of Conjecture~\ref{conjecture degree L} holds.
    Additionally, in Theorem~\ref{theorem quadratic forward direction} we will show that when $\ell=2$, for ``many'' graphs, the forward direction of Conjecture~\ref{conjecture degree L} holds.
    
    Before we present these results, let us give some intuition on Conjecture~\ref{conjecture degree L} and the parameters $\alpha_{s,\ell}(k,i)$.
    Recall from Theorem~\ref{theorem linear algebraic characterization of copolynimal graphs} that the sum of the form $\sum_{W\in \binom{S}{k}}|\mathrm{Obs}(G,W)|$ corresponds exactly to the coefficient of the basis vector $q^{s-k}(1-q)^k$ when we write $\expolq{G}{S}$ in terms of the basis $\{q^{s-k}(1-q)^k:0\leq k\leq s\}$.
    In this way, Conjecture~\ref{conjecture degree L} can be thought of as saying that when we restrict our attention to polynomials of degree at most $\ell$ (where $\ell<s$), the loss of degrees of freedom caused by this restriction corresponds exactly to the basis vectors $q^{s-k}(1-q)^k$ with $k>\ell$.
    
    Let us consider a guiding example to get a sense of the $\alpha_{s,\ell}(k,i)$'s.
    If we have a graph $G$, sensor locations $S\subseteq V(G)$, and $\expolq{G}{S}$ is quadratic (but $|S|=s>2$), then we expect the sum corresponding to triples, $\sum_{W\in \binom{S}{3}}|\mathrm{Obs}(G,W)|,$ should be determined solely based on the sums involving sets of size $2$ or less.
    In particular Conjecture~\ref{conjecture degree L} claims that
     \[
        \sum_{W\in \binom{S}{3}}|\mathrm{Obs}(G,W)|=\binom{s-2}{3-2}\sum_{W\in \binom{S}{2}}|\mathrm{Obs}(G,W)|+\left(\binom{s-1}{3-1}-\binom{s-2}{3-2}\binom{s-1}{2-1}\right)\sum_{v\in S}|\mathrm{Obs}(G,W)|
    \]
    The idea behind these coefficients is that contribution of pairs to the sum $\sum_{W\in \binom{S}{3}}|\mathrm{Obs}(G,W)|$ should be $s-2=\binom{s-2}{3-2}$ since each pair in $S$ is contained in $s-2$ triples in $S$, giving us $\alpha_{s,2}(3,2)=\binom{s-2}{3-2}$.
    Similarly, the contribution of singletons in $S$ to the sum $\sum_{W\in \binom{S}{3}}|\mathrm{Obs}(G,W)|$ should be $\binom{s-1}{2}=\binom{s-1}{3-1}$ since each singleton is contained in $\binom{s-1}{2}$ triples; however some of the contribution of singletons is already embedded in the contribution we've already accounted for from the pairs.
    In particular, each singleton in is $s-1=\binom{s-1}{2-1}$ pairs, and each pair has already been counted $\binom{s-2}{3-2}$ times in the term $\binom{s-2}{3-2}\sum_{W\in \binom{S}{2}}|\mathrm{Obs}(G,W)|$.
    We then set $\alpha_{s,3}(3,1)=\binom{s-1}{3-1}-\binom{s-2}{3-2}\binom{s-1}{2-1}=\binom{s-1}{3-1}-\alpha_{s,3}(3,2)\binom{s-1}{2-1}$ so that we get a total singleton contribution of $\binom{s-1}{3-1}$ without over counting.
    When $\ell,k>3$ the $\alpha$'s are defined similarly to make it so that $i$-sets contribute a total of $\binom{s-i}{k-i}$, after accounting for the contributions embedded in the terms corresponding to $j$-sets for $i<j\leq \ell$. 
    
    We now turn our attention to proving the reverse direction of Conjecture~\ref{conjecture degree L}, beginning with a lemma.
    
    \begin{lemma}\label{lemma alphas are degree at most L}
        Let $\ell\leq s$ be natural numbers.
        Let $\alpha_{s,\ell}(k,i)$ be defined as in~\eqref{equation definition of the alphas}.
        Then $\alpha_{s,\ell}(k,i)$ is well-defined for all $k,i\in\mathbb{N}$ and
        \[
            \sum_{k=1}^s\alpha_{s,\ell}(k,i)q^{s-k}(1-q)^k
        \]
        is a polynomial in $q$ of degree at most $\ell$ for all $i\leq \ell$.
    \end{lemma}
    
    \begin{proof}
        Let us fix $\ell$ and $s$ and write $\alpha(k,i):=\alpha_{s,\ell}(k,i)$ for ease of notation.
        If we fix $k$, then note that for all $i\geq \ell$, we have $\alpha(k,i)=\binom{s-i}{k-i}$, and then for $i<\ell$,  we can see that $\alpha(k,i)$ only depends on $\alpha(k,j)$ for $j>i$, and thus by strong induction on $\ell-i$ (note that $j>i$ implies that $\ell-j<\ell-i$), $\alpha(k,i)$ is well-defined.
        
        Now, consider $\sum_{k=1}^s\alpha(k,i)q^{s-k}(1-q)^k$.
        Let us proceed by induction on $\ell-i$.
        When $i=\ell$, we have
        \[
            \sum_{k=1}^s\alpha(k,i)q^{s-k}(1-q)^k=\sum_{k=1}^s\binom{s-i}{k-i}q^{s-k}(1-q)^k=(1-q)^i\sum_{k=1}^s\binom{s-i}{k-i}q^{s-k}(1-q)^{k-i}=(1-q)^i,
        \]
        where the last equality follows from the Binomial Theorem.
        Since $i=\ell$, the resulting polynomial is degree $\ell$.
        
        Now let $\ell-i>0$ and assume $\sum_{k=1}^s\alpha(k,i^*)q^{s-k}(1-q)^k$ is a polynomial of degree at most $\ell$ whenever $\ell-i^*<\ell-i$ (i.e. whenever $i^*>i$).
        Then
        \begin{align*}
            \sum_{k=1}^s\alpha(k,i)q^{s-k}(1-q)^k&=\sum_{k=1} ^{s}\left(\binom{s-i}{k-i}-\sum_{j=i+1}^\ell \alpha(k,j)\cdot \binom{s-i}{j-i}\right)q^{s-k} (1-q)^{k}\\
            &=\sum_{k=1} ^{s}\binom{s-i}{k-i}q^{s-k} (1-q)^{k}-\sum_{k=1}^s\sum_{j=i+1}^\ell \alpha(k,j)\cdot \binom{s-i}{j-i}q^{s-k} (1-q)^{k}\\
            &=(1-q)^i\sum_{k=1} ^{s}\binom{s-i}{k-i}q^{s-k} (1-q)^{k-i}-\sum_{j=i+1}^\ell\binom{s-i}{j-i}\sum_{k=1}^s \alpha(k,j)\cdot q^{s-k} (1-q)^{k}\\
            &=(1-q)^i-\sum_{j=i+1}^\ell\binom{s-i}{j-i}\sum_{k=1}^s \alpha(k,j)\cdot q^{s-k} (1-q)^{k}.
        \end{align*}
        Then since $i<\ell$, $(1-q)^i$ is degree at most $\ell$, and by induction, $\sum_{k=1}^s \alpha(k,j)\cdot q^{s-k} (1-q)^{k}$ is degree at most $\ell$, so the sum of these is degree at most $\ell$.
    \end{proof}
    
    The following theorem proves the reverse direction of Conjecture~\ref{conjecture degree L}.
    
    \begin{theorem}
        Given a graph $G$ and a set $S\subseteq V(G)$ with $|S|=s$.
        If for all $1\leq k\leq s$
        \begin{equation}\label{equation degree iff condition}
            \sum_{W\in \binom{S}{k}} \sizeobs{G}{W}=\sum_{i=1}^{\ell}\alpha_{s,\ell}(k,i)\sum_{W\in \binom{S}{i}} \sizeobs{G}{W},
        \end{equation}
        where $\alpha_{s,\ell}(k,i)$ is as in~\eqref{equation definition of the alphas}, then $\mathcal{E}(G,S,q)$ is degree at most $\ell$.
    \end{theorem}
    
    \begin{proof}
        If \eqref{equation degree iff condition} is satisfied for all $1\leq k \leq s$, then observe that
            \begin{align*}
                \expolq{G}{S}   &= \sum_{W\subseteq S} \sizeobs{G}{W} q^{|S\setminus W|} (1-q)^{|W|} \\
                                &= \sum_{k=1} ^{s} \sum_{W\in \binom{S}{k}} \sizeobs{G}{W} q^{s-k} (1-q)^{k} \\
                                &= \sum_{k=1} ^{s}\sum_{i=1}^{\ell}\alpha_{s,\ell}(k,i)\left(\sum_{W\in \binom{S}{i}} \sizeobs{G}{W}\right)q^{s-k} (1-q)^{k} \\
                                &=\sum_{i=1}^{\ell}\left(\sum_{W\in \binom{S}{i}} \sizeobs{G}{W}\right)\sum_{k=1} ^{s}\alpha_{s,\ell}(k,i)q^{s-k} (1-q)^{k}.
            \end{align*}
        Then by Lemma~\ref{lemma alphas are degree at most L}, $\sum_{k=1} ^{s}\alpha_{s,\ell}(k,i)q^{s-k} (1-q)^{k}$ is a polynomial of degree at most $\ell$, so $\expolq{G}{S}$ is degree at most $\ell$.
    \end{proof}
    
    We close this section by proving the forward direction of Conjecture~\ref{conjecture degree L} holds for graphs $G$ and sensor locations $S\subseteq V(G)$, as long as the linear coefficient in $\expolq{G}{S}$ is a large enough negative.
    Before we state the specific result, we note a few things about this restriction on the linear coefficient.
    First, recall that Proposition~\ref{proposition uncontrollable conefficients} gives us that $\expolq{G}{S}[q]\leq 0$, so the condition that the linear coefficient is a large negative is not as restrictive as it may seem.
    Second, the idea behind the proof of Theorem~\ref{theorem quadratic forward direction} is similar to that of Theorem~\ref{theorem linear} --- given a graph $G$ and sensor location $S$, we wish to replace $G$ with some other graph $H$ with the same expected value polynomial, but whose structure we understand better.
    We then want to show that $H$ satisfies~\eqref{equation quadratic expression for lambdak}, and using Theorem~\ref{theorem linear algebraic characterization of copolynimal graphs}, lift this back to a result for $G$.
    Theorem~\ref{theorem induced any polynomial} allows us to construct a graph $H$ that matches most of the coefficients of $\expolq{G}{S}$, however in Theorem~\ref{theorem induced any polynomial}, we do not have complete control over the linear coefficient.
    It is straightforward to add structure which will make the linear coefficient more negative without affecting the other coefficients; one can add multiple leaves adjacent to a particular sensor location to affect only the linear coefficient. 
    Adding structures in this way allows us to make the following claim about all polynomials with sufficiently large (negative) linear coefficient.
    
    \begin{theorem}\label{theorem quadratic forward direction}
        For any integers $s\in\mathbb{N}$ and $a\in \mathbb{Z}$ and any graph $G$ and set $S\subseteq V(G)$ with $|S|=s$, if $b\geq \max\{s+2a+16,s+6\}$ and
        \[
            \expolq{G}{S}=ax^2-bx+c,
        \]
        then for $1\leq k\leq s$,
        \begin{equation}\label{equation quadratic expression for lambdak}
            \sum_{W\in \binom{S}{k}}|\mathrm{Obs}(G,W)|=\binom{s-2}{k-2}\sum_{W\in \binom{S}{2}}|\mathrm{Obs}(G,W)|+\left(\binom{s-1}{k-1}-(s-1)\binom{s-2}{k-2}\right)\sum_{v\in S}|\mathrm{Obs}(G,W)|
        \end{equation}
    \end{theorem}
    \begin{figure}
        \centering
        \begin{tikzpicture}
            \node [style=Black Node] (1) at (1, 1) {};
            \node [style=Black Node] (2) at (-1, 1) {};
            \node [style=Black Node] (4) at (1, -1) {};
            \node [style=Black Node] (5) at (-1, -1) {};
            \node [style=Black Node] (6) at (0, 0) {};
            \node [style=Black Node] (7) at (1, 0) {};
            \node [style=Red Node] (8) at (2, 0) {};
            \node [style=Blue Node] (9) at (0, -1) {};
            \node [style=Blue Node] (10) at (0, 1) {};
            \node [style=Black Node] (11) at (-1, 0.5) {};
            \node [style=Black Node] (12) at (-1, -0.5) {};
            \node [style=Black Node] (13) at (-2, 0) {};
            \node [style=Red Node] (14) at (-3, 0) {};
            \node [style=Black Node] (24) at (3, 0) {};
            \node [style=Black Node] (26) at (-4, 0) {};
            \node [style=Text Node] (27) at (3.5, -0.25) {$P_w$};
            \node [style=Text Node] (28) at (-4.5, -0.25) {$P_t$};
            \node [style=Black Node] (29) at (4, 0) {};
            \node [style=Black Node] (30) at (-5, 0) {};
            \draw [style=Black Edge] (8) to (7);
            \draw [style=Black Edge] (7) to (10);
            \draw [style=Black Edge] (10) to (6);
            \draw [style=Black Edge] (6) to (9);
            \draw [style=Black Edge] (9) to (7);
            \draw [style=Black Edge] (1) to (10);
            \draw [style=Black Edge] (10) to (2);
            \draw [style=Black Edge] (4) to (9);
            \draw [style=Black Edge] (9) to (5);
            \draw [style=Black Edge] (10) to (11);
            \draw [style=Black Edge] (11) to (13);
            \draw [style=Black Edge] (13) to (12);
            \draw [style=Black Edge] (12) to (9);
            \draw [style=Black Edge] (14) to (13);
            \draw [style=Black Edge] (26) to (14);
            \draw [style=Black Edge] (8) to (24);
            \draw [style=Dashed Edge] (24) to (29);
            \draw [style=Dashed Edge] (30) to (26);
            \draw[rounded corners=10pt, style=Black Edge] (2.75, -0.5) rectangle (4.25, 0.5);
            \draw[rounded corners=10pt, style=Black Edge] (-3.75, -0.5) rectangle (-5.25, 0.5);
        \end{tikzpicture}
        \caption{The graph $H'(t,w):=\left(2P_3\boxminus_{A}\forka{t}{2}\right)\boxminus_{A}\spoona{w}{2}$ with path head vertices of $\forka{t}{2}$ and $\spoona{w}{2}$ indicated in red, and $A$ indicated in blue. Dashed edges represent the appended paths with $t$ and $w$ vertices.}
        \label{fig:hprime}
    \end{figure}
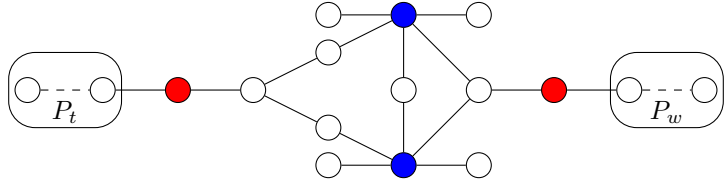

    \begin{proof}
        Note that if $k\in \{1,2\}$,~\eqref{equation quadratic expression for lambdak} simplifies down to $\sum_{W\in \binom{S}{k}}|\mathrm{Obs}(G,W)|=\sum_{W\in \binom{S}{k}}|\mathrm{Obs}(G,W)|$, thus we may assume that $k\geq 3$, and thus further that $s\geq 3$.
        Furthermore, note that since $0=\expol{G}{S}{1}=a-b+c$, we must have $c=b-a$. 
        
        Consider the graph $2P_3$ with $A\subseteq V(2P_3)$ the set consisting of the two vertices of degree $2$.
        Then let $H'(t,w):=\left(2P_3\boxminus_{A}\forka{t}{2}\right)\boxminus_{A}\spoona{w}{2}$ as in Figure~\ref{fig:hprime}.
        Then let $H=H(t,w,d):=H'(t,w)\sqcup K_{d+1}\sqcup\overline{K_{s-3}}$.
        Finally let $S'\subseteq V(H)$ denote the set of the $s-3$ isolated vertices, along with the two vertices in $A$ and one vertex arbitrarily chosen from the $K_{d+1}$.
        It is straightforward to verify that for $t,w,d\geq 0$,
        \begin{equation}\label{equation polynomial for Htwd}
            \expolq{H}{S'}=(t-w-5)q^2-(s+2t+d+6)q+(s+t+w+d+11).
        \end{equation}
    
        We can choose parameters $t,w,d$ so that $\expolq{G}{S}=\expolq{H}{S'}$.
        Indeed, if $a\leq -5$, we can choose $t=0$, $w=-a-5$ and $d=b-(s+6)$.
        If instead $a\geq -4$, then we set $t=a+5$, $w=0$ and $d=b-(s+2a+16)$.
        In either case it is straightforward to verify in~\eqref{equation polynomial for Htwd} that this choice of parameters gives $aq^2-bq+c$, and that $t,w,d\geq 0$ (using the hypothesis that $b\geq \max\{s+2a+16,s+6\}$).
    
        Thus, by Theorem~\ref{theorem linear algebraic characterization of copolynimal graphs}, since $\expolq{G}{S}=\expolq{H}{S'}$ we have that for all $1\leq k\leq s$,
        \[
            \sum_{W\in \binom{S}{k}}|\mathrm{Obs}(G,W)|=\sum_{W\in \binom{S'}{k}}|\mathrm{Obs}(H,W)|,
        \]
        so it will suffice to show that~\eqref{equation quadratic expression for lambdak} holds for $H$.
        
        Let us write $H_1,H_2,\dots,H_{s-1}$ for the connected components of $H$, where we will assume $H_1=H'(t,w)$, $H_2=K_d$, and the remaining $H_i$'s are isolated vertices.
        Finally, let $S':=\{v_0,v_1,\dots,v_{s-1}\}$, where $v_0,v_1\in A\subseteq V(H_1)$, and then $v_j\in V(H_j)$ for $j\geq 2$.
        
        We note that for any $W\subseteq S'$,
        \[
            |\mathrm{Obs}(H;W)|=\sum_{i=1}^{s-1}|\mathrm{Obs}(H_i;W\cap V(H_i))|
        \]
        For a fixed $k\geq 1$, let us calculate
        \begin{align}
            \sum_{W\in \binom{S'}{k}}|\mathrm{Obs}(H;W)|&=\sum_{W\in \binom{S'}{k}}\sum_{i=1}^{s-1}|\mathrm{Obs}(H_i;W\cap V(H_i))|\nonumber\\
            &=\sum_{i=1}^{s-1}\sum_{W\in \binom{S'}{k}}|\mathrm{Obs}(H_i;W\cap V(H_i))|\nonumber\\
            &=\sum_{W\in \binom{S'}{k}}|\mathrm{Obs}(H_1;W\cap V(H_1))|+\sum_{i=2}^{s-1}\sum_{W\in \binom{S'}{k}}|\mathrm{Obs}(H_i;W\cap V(H_i))|\label{equation breaking up the sum of W by components}
        \end{align}
        For the first sum above, $\displaystyle\sum_{W\in \binom{S'}{k}}|\mathrm{Obs}(H_1;W\cap V(H_1))|$, we note that $W\cap V(H_1))=A$ for exactly $\binom{s-2}{k-2}$ $W$'s.
        Also, we have $W\cap V(H_1))=\{v_0\}$ for exactly $\binom{s-2}{k-1}$ $W$'s, and the same occurs when $W\cap V(H_1))=\{v_1\}$.
        Furthermore, these $W$'s are all distinct, and the $W$'s considered here correspond to all the $W$'s which give a non-zero contribution to $|\mathrm{Obs}(H_1;W\cap V(H_1))|$.
        Thus, we may write
        \begin{align*}
            \sum_{W\in \binom{S'}{k}}|\mathrm{Obs}(H_1;W\cap V(H_1))|&=\binom{s-2}{k-2}|\mathrm{Obs}(H_1;A)|+\binom{s-2}{k-1}\left(|\mathrm{Obs}(H_1;v_0)|+|\mathrm{Obs}(H_1;v_1)|\right)\\
            &=\binom{s-2}{k-2}|\mathrm{Obs}(H_1;A)|+\left(\binom{s-1}{k-1}-\binom{s-2}{k-2}\right)\left(|\mathrm{Obs}(H_1;v_0)|+|\mathrm{Obs}(H_1;v_1)|\right)\\
        \end{align*}
        Now, for the second sum in~\eqref{equation breaking up the sum of W by components}, if we fix some $i\geq 2$, we have that there are exactly $\binom{s-1}{k-1}$ choices for $W$ for which $v_i\in W$, and in all other cases $|\mathrm{Obs}(H_i;W\cap V(H_i))|=0$, and thus
        \[
            \sum_{W\in \binom{S'}{k}}|\mathrm{Obs}(H_i;W\cap V(H_i))|=\binom{s-1}{k-1}|\mathrm{Obs}(H_i;v_i)|
        \]
        Putting this together and returning to~\eqref{equation breaking up the sum of W by components}, we have
        \begin{align}
            &\sum_{W\in \binom{S'}{k}}|\mathrm{Obs}(H;W)|\nonumber\\
            &=\binom{s-2}{k-2}|\mathrm{Obs}(H_1;A)|+\left(\binom{s-1}{k-1}-\binom{s-2}{k-2}\right)\left(|\mathrm{Obs}(H_1;v_1)|+|\mathrm{Obs}(H_1;v_0)|\right)+\sum_{i=2}^{s-1}\binom{s-1}{k-1}|\mathrm{Obs}(H_i;v_i)|\nonumber\\
            &=\binom{s-2}{k-2}|\mathrm{Obs}(H_1;A)|-\binom{s-2}{k-2}\left(|\mathrm{Obs}(H_1;v_1)|+|\mathrm{Obs}(H_1;v_0)|\right)+\binom{s-1}{k-1}\sum_{v\in S'}|\mathrm{Obs}(H,v)|\nonumber\\
            &=\binom{s-2}{k-2}\left(|\mathrm{Obs}(H_1;A)|-|\mathrm{Obs}(H;v_1)|-|\mathrm{Obs}(H;v_0)|\right)+\binom{s-1}{k-1}\sum_{v\in S'}|\mathrm{Obs}(H,v)|.\label{equation intermediate equation 1}
        \end{align}
    
        Now consider the sum below (using~\eqref{equation breaking up the sum of W by components}).
    
        \begin{align*}
            \sum_{W\in \binom{S'}{2}}|\mathrm{Obs}(H;W)|&=\sum_{i=1}^{s-1}\sum_{W\in \binom{S'}{2}}|\mathrm{Obs}(H_i;W\cap V(H_i))|.
        \end{align*}
        There is exactly one set $W$ and choice of $i\in [s-1]$ such that $|W\cap V(H_i)|=2$, namely $W=A$ and $i=1$.
        In all other non-zero cases, we have $|W\cap V(H_i)|=1$, and in particular, for $v_0$ and $v_1$, there are exactly $s-2$ instances where $W\cap V(H_i)$ is $\{v_0\}$ or $\{v_1\}$.
        For all $j\geq 2$, there are $s-1$ instances where $W\cap V(H_i)=\{v_j\}$.
        This gives us that 
        \begin{align*}
            \sum_{W\in \binom{S'}{2}}|\mathrm{Obs}(H;W)|&=\sum_{i=1}^{s-1}\sum_{W\in \binom{S'}{2}}|\mathrm{Obs}(H_i;W\cap V(H_i))|\\
            &=|\mathrm{Obs}(H,A)|+(s-2)\sum_{i=0}^1|\mathrm{Obs}(H,v_i)|+(s-1)\sum_{i=2}^{s-1}|\mathrm{Obs}(H,v_i)|\\
            &=|\mathrm{Obs}(H,A)|-|\mathrm{Obs}(H,v_0)|-|\mathrm{Obs}(H,v_1)|+(s-1)\sum_{v\in S'}|\mathrm{Obs}(H,v)|\\
        \end{align*}
        Rearranging and then multiplying by $\binom{s-2}{k-2}$, we obtain
        \begin{equation}\label{equation intermediate equation 2}
            \binom{s-2}{k-2}(|\mathrm{Obs}(H,A)|-|\mathrm{Obs}(H,v_0)|-|\mathrm{Obs}(H,v_1)|)=\binom{s-2}{k-2}\left(\sum_{W\in \binom{S'}{2}}|\mathrm{Obs}(H;W)|-(s-1)\sum_{v\in S'}|\mathrm{Obs}(H,v)|\right)
        \end{equation}
        Returning to~\eqref{equation intermediate equation 1}, using~\eqref{equation intermediate equation 2}, yields
    
        \begin{align*}
            \sum_{W\in \binom{S'}{k}}|\mathrm{Obs}(H;W)|&=\binom{s-2}{k-2}\left(|\mathrm{Obs}(H_1;A)|-|\mathrm{Obs}(H;v_1)|-|\mathrm{Obs}(H;v_0)|\right)+\binom{s-1}{k-1}\sum_{v\in S'}|\mathrm{Obs}(H,v)|\\
            &=\binom{s-2}{k-2}\left(\sum_{W\in \binom{S'}{2}}|\mathrm{Obs}(H;W)|-(s-1)\sum_{v\in S'}|\mathrm{Obs}(H,v)|\right)+\binom{s-1}{k-1}\sum_{v\in S'}|\mathrm{Obs}(H,v)|\\
            &=\binom{s-2}{k-2}\sum_{W\in \binom{S'}{2}}|\mathrm{Obs}(H;W)|+\left(\binom{s-1}{k-1}-(s-1)\binom{s-2}{k-2}\right)\sum_{v\in S'}|\mathrm{Obs}(H,v)|.
        \end{align*}
    
        Thus,~\eqref{equation quadratic expression for lambdak} holds for $H$, and thus also for $G$.
    \end{proof}
    
    We believe that with the tools we've developed in this work, Theorem~\ref{theorem quadratic forward direction} can be generalized to any degree $\ell$.
    More specifically, we believe that for any $s\geq \ell\geq 2$, and any values $c_\ell,c_{\ell-1},\dots,c_2$, there exists a function $f_{\ell}(s,c_\ell,c_{\ell-1},\dots,c_2)$ such that for any graph $G$ and sensor set $S\subseteq V(G)$ with $|S|=s$, if $\expolq{G}{S}$ has degree at most $\ell$, and if $\expolq{G}{S}[q^i]=c_i$ for all $2\leq i\leq \ell$, then as long as $\expolq{G}{S}[q]\leq f_{\ell}(s,c_\ell,c_{\ell-1},\dots,c_2)$, $G$ and $S$ will satisfy~\eqref{equation condition for degree at most L}.
    The idea is to create $H$ based on the gadgets we developed in Section~\ref{sec:gadgets} to match $\expolq{G}{S}$, and otherwise follow the outline given in the proof of Theorem~\ref{theorem quadratic forward direction}.
    Since this is not sufficient to prove Conjecture~\ref{conjecture degree L} in full generality, we do not formally claim this result here to avoid further muddling this paper with a highly technical, if partial, result.
        
\section{Comparing sensor locations}\label{sec:highdeg}

    Given two polynomials $p_1$ and $p_2$ on $[0,1]$, we will write $p_1\preceq p_2$ if for all $q\in [0,1]$, $p_1(q)\leq p_2(q)$.
    We will write $p_1\prec p_2$ if $p_1\preceq p_2$ and $p_1\neq p_2$.
    If we have two expected value polynomials, $p_1$ and $p_2$, satisfying $p_1 \preceq p_2$, we can see that the choice of sensor locations corresponding to $p_2$ allows for better (or at least the same) network coverage at every possible probability of PMU failure. 
    In the following proposition, we present a method which demonstrates when certain sets of sensor locations are strictly better for fragile power domination.
    
    \begin{proposition}\label{lemma sums imply dominating polynomial}
        Let $G$ be a graph and let $A,B\subseteq V(G)$ with $|A|=|B|=:a$.
        If
        \begin{equation}\label{equation condition for bijection domination result}
            \sum_{A'\in \binom{A}{k}}\sizeobs{G}{A'}\leq \sum_{B'\in \binom{B}{k}}\sizeobs{G}{B'}
        \end{equation}
        for all $k\in [a]$, then
        \begin{equation}\label{equation consequence of domination result}
            \expolq{G}{A}\preceq\expolq{G}{B}.
        \end{equation}
        Moreover, if there exists some $k\in [a]$ where~\eqref{equation condition for bijection domination result} is strict, then~\eqref{equation consequence of domination result} is strict as well.
    \end{proposition}
    
    \begin{proof}
        Fix $q\in [0,1]$.
        Starting from~\eqref{equation set view of E}, and
        using~\eqref{equation condition for bijection domination result} we can write
        \begin{align}
            \expolq{G}{A}&=\sum_{A'\subseteq A}|\obs{G}{A'}|q^{|A\setminus A'|}(1-q)^{|A'|}\\
            &=\sum_{k=1}^{a}\sum_{A'\in\binom{A}{k}}\sizeobs{G}{A'}q^{a-k}(1-q)^{k}\notag\\
            &\leq\sum_{k=1}^{a}\sum_{B'\in\binom{B}{k}}\sizeobs{G}{B'}q^{a-k}(1-q)^{k}=\expolq{G}{B}.\label{equation domination result strict or not strict}
        \end{align}
        Thus, $\expolq{G}{A}\preceq\expolq{G}{B}$. 
        
        Furthermore, for any $q\in (0,1)$ and any $k\in [a]$, $q^{a-k}(1-q)^{k}\neq 0$, and thus as long as there exists some $k\in [a]$ such that~\eqref{equation condition for bijection domination result} is strict,~\eqref{equation domination result strict or not strict} is also strict for $q\in (0,1)$, and consequentially~\eqref{equation consequence of domination result} is strict.
    \end{proof}
    
    Consider the relationship between Proposition~\ref{lemma sums imply dominating polynomial} and Theorem~\ref{theorem linear algebraic characterization of copolynimal graphs}.
    Theorem~\ref{theorem linear algebraic characterization of copolynimal graphs} shows that given a graph $G$ and a set $S$, sums of the form $\sum_{S' \in \binom {S}{k}}\sizeobs{G}{S}$ are the parameters which dictate the specific polynomial we get for $\expolq{G}{S}$.
    Proposition~\ref{lemma sums imply dominating polynomial} further shows that there is a monotonic relationship between these sums and the behavior of the polynomial for $q\in [0,1]$.

    It was observed in~\cite{HHHH2002} that in graphs $G$ with at least one vertex with degree 3 or higher, there always exists a minimum power dominating set which does not contain vertices of degree $1$ or $2$.
    A natural heuristic in power domination is that high degree vertices tend to be more helpful as they observe more in the domination step.
    However, we will demonstrate construction of a graph $G$ for which a minimum power dominating set containing only vertices of degree $3$ is strictly better for fragile power domination than any set containing even one sensor on a relatively high degree vertex.

    \begin{construction}\label{construction max degree}
        Let $s\geq 4$ be even and $\ell\in\mathbb{N}$.
        Set $\eta:=\ell+2s$.
        Let $\mathcal{M}=\{M_1,M_2,\dots,M_{s-1}\}$ be a $1$-factorization of $K_s$ with vertex set $V(K_s)=[s]$.
        Let $G_s(\ell)$ denote the graph on $s\cdot\eta$ vertices formed in the following way.
        Let \[V(G_s(\ell)):=\{v_{i,j} \setgivensymbol i\in [s], j\in[ \eta]\}.\]
        Construct the following sets of edges: 
        \begin{align*}
            E_1 &= \{ v_{i,j}v_{i,j+1} : i\in [ s], j\in[ \eta-1] \}, \\
            E_2 &= \{v_{i,2}v_{i',2} : i,i' \in [s]\}, \text{ and}\\
            E_3 &= \{v_{i,2j+2}v_{i',2j+2} : j\in [s-1], ii'\in M_j\}.
        \end{align*}
        Finally, let $E(G_s(\ell)) = E_1 \cup E_2 \cup E_3$.
        See Figure~\ref{figure max degree} for a drawing of $G_4(4)$.
        
        For each $i\in [s]$, we define the $i$th \emph{row}, $R_i:=\{v_{i,j} \setgivensymbol j\in [\eta]\}$, and for each $j\in [\eta]$, we define the $j$th \emph{column}, $C_j:=\{v_{i,j} \setgivensymbol i\in [s]\}$.
        In addition, for $i\in [s]$, we define the $i$th \emph{pendant path}, $L_i:=\{v_{i,j} \setgivensymbol j\geq 2s+1\}$ and the $i$th \emph{fort}, $F_i:=\{v_{i,2j-1} \setgivensymbol j \in [s]\}\cup L_i$. 
        Finally, the \emph{clique vertices} will be $K:=C_{2}$ and the \emph{Special vertices} will be $S:=C_{2s}$.
    \end{construction}
    
    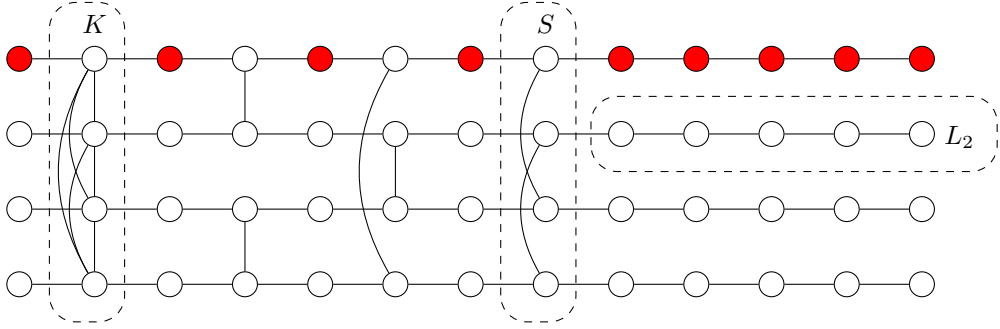
\begin{figure}
        \centering
        \begin{tikzpicture}
            \node [style=Black Node] (0) at (0, 0) {};
            \node [style=Black Node] (1) at (1, 0) {};
            \node [style=Black Node] (2) at (2, 0) {};
            \node [style=Black Node] (3) at (3, 0) {};
            \node [style=Black Node] (4) at (4, 0) {};
            \node [style=Black Node] (5) at (5, 0) {};
            \node [style=Black Node] (6) at (6, 0) {};
            \node [style=Black Node] (7) at (7, 0) {};
            \node [style=Black Node] (8) at (8, 0) {};
            \node [style=Black Node] (9) at (9, 0) {};
            \node [style=Black Node] (10) at (10, 0) {};
            \node [style=Black Node] (11) at (11, 0) {};
            \node [style=Black Node] (12) at (12, 0) {};
            \node [style=Black Node] (13) at (0, 1) {};
            \node [style=Black Node] (14) at (1, 1) {};
            \node [style=Black Node] (15) at (2, 1) {};
            \node [style=Black Node] (16) at (3, 1) {};
            \node [style=Black Node] (17) at (4, 1) {};
            \node [style=Black Node] (18) at (5, 1) {};
            \node [style=Black Node] (19) at (6, 1) {};
            \node [style=Black Node] (20) at (7, 1) {};
            \node [style=Black Node] (21) at (8, 1) {};
            \node [style=Black Node] (22) at (9, 1) {};
            \node [style=Black Node] (23) at (10, 1) {};
            \node [style=Black Node] (24) at (11, 1) {};
            \node [style=Black Node] (25) at (12, 1) {};
            \node [style=Black Node] (26) at (0, 2) {};
            \node [style=Black Node] (27) at (1, 2) {};
            \node [style=Black Node] (28) at (2, 2) {};
            \node [style=Black Node] (29) at (3, 2) {};
            \node [style=Black Node] (30) at (4, 2) {};
            \node [style=Black Node] (31) at (5, 2) {};
            \node [style=Black Node] (32) at (6, 2) {};
            \node [style=Black Node] (33) at (7, 2) {};
            \node [style=Black Node] (34) at (8, 2) {};
            \node [style=Black Node] (35) at (9, 2) {};
            \node [style=Black Node] (36) at (10, 2) {};
            \node [style=Black Node] (37) at (11, 2) {};
            \node [style=Black Node] (38) at (12, 2) {};
            \node [style=Black Node] (39) at (1, 3) {};
            \node [style=Black Node] (40) at (3, 3) {};
            \node [style=Black Node] (41) at (5, 3) {};
            \node [style=Black Node] (42) at (7, 3) {};
            \node [style=Red Node] (43) at (0, 3) {};
            \node [style=Red Node] (44) at (2, 3) {};
            \node [style=Red Node] (45) at (4, 3) {};
            \node [style=Red Node] (46) at (6, 3) {};
            \node [style=Red Node] (47) at (8, 3) {};
            \node [style=Red Node] (48) at (9, 3) {};
            \node [style=Red Node] (49) at (10, 3) {};
            \node [style=Red Node] (50) at (11, 3) {};
            \node [style=Red Node] (51) at (12, 3) {};
            \node [style=Text Node] (52) at (1, 3.5) {$K$};
            \node [style=Text Node] (53) at (7, 3.5) {$S$};
            \node [style=Text Node] (54) at (12.5, 2) {$L_2$};
            \draw [style=Black Edge] (39) to (27);
            \draw [style=Black Edge] (27) to (14);
            \draw [style=Black Edge] (14) to (1);
            \draw [style=Black Edge, bend left] (1) to (27);
            \draw [style=Black Edge, bend right] (39) to (14);
            \draw [style=Black Edge, bend right] (39) to (1);
            \draw [style=Black Edge] (40) to (29);
            \draw [style=Black Edge] (16) to (3);
            \draw [style=Black Edge] (43) to (39);
            \draw [style=Black Edge] (27) to (26);
            \draw [style=Black Edge] (13) to (14);
            \draw [style=Black Edge] (1) to (0);
            \draw [style=Black Edge] (1) to (2);
            \draw [style=Black Edge] (2) to (3);
            \draw [style=Black Edge] (3) to (4);
            \draw [style=Black Edge] (4) to (5);
            \draw [style=Black Edge] (5) to (6);
            \draw [style=Black Edge] (6) to (7);
            \draw [style=Black Edge] (7) to (8);
            \draw [style=Black Edge] (8) to (9);
            \draw [style=Black Edge] (9) to (10);
            \draw [style=Black Edge] (10) to (11);
            \draw [style=Black Edge] (11) to (12);
            \draw [style=Black Edge] (25) to (24);
            \draw [style=Black Edge] (24) to (23);
            \draw [style=Black Edge] (23) to (22);
            \draw [style=Black Edge] (22) to (21);
            \draw [style=Black Edge] (21) to (20);
            \draw [style=Black Edge] (20) to (19);
            \draw [style=Black Edge] (19) to (18);
            \draw [style=Black Edge] (18) to (17);
            \draw [style=Black Edge] (17) to (16);
            \draw [style=Black Edge] (16) to (15);
            \draw [style=Black Edge] (15) to (14);
            \draw [style=Black Edge] (27) to (28);
            \draw [style=Black Edge] (28) to (29);
            \draw [style=Black Edge] (29) to (30);
            \draw [style=Black Edge] (30) to (31);
            \draw [style=Black Edge] (31) to (32);
            \draw [style=Black Edge] (32) to (33);
            \draw [style=Black Edge] (33) to (34);
            \draw [style=Black Edge] (34) to (35);
            \draw [style=Black Edge] (35) to (36);
            \draw [style=Black Edge] (36) to (37);
            \draw [style=Black Edge] (37) to (38);
            \draw [style=Black Edge] (51) to (50);
            \draw [style=Black Edge] (50) to (49);
            \draw [style=Black Edge] (49) to (48);
            \draw [style=Black Edge] (48) to (47);
            \draw [style=Black Edge] (47) to (42);
            \draw [style=Black Edge] (42) to (46);
            \draw [style=Black Edge] (46) to (41);
            \draw [style=Black Edge] (41) to (45);
            \draw [style=Black Edge] (45) to (40);
            \draw [style=Black Edge] (40) to (44);
            \draw [style=Black Edge] (44) to (39);
            \draw [style=Black Edge, bend right] (42) to (20);
            \draw [style=Black Edge, bend right] (33) to (7);
            \draw [style=Black Edge] (31) to (18);
            \draw [style=Black Edge, bend right] (41) to (5);
            \draw[rounded corners=10pt, style=Dashed Edge] (0.4, -0.5) rectangle (1.4, 3.75);
            \draw[rounded corners=10pt, style=Dashed Edge] (6.4, -0.5) rectangle (7.4, 3.75);
            \draw[rounded corners=10pt, style=Dashed Edge] (7.6, 1.5) rectangle (13, 2.5);
        \end{tikzpicture}
        \caption{The graph $G_4(5)$ from Construction~\ref{construction max degree} with $F_1$ indicated in red.}\label{figure max degree}
    \end{figure}
    
    As we will see in the result below, as long as $\ell$ is large enough, the set $S$ is a strictly better location for sensors than any set of size $\gamma_P(G_s(\ell))$ containing a vertex from $K$.
    This defies the high degree heuristic as vertices in $K$ have degree $\binom{s}2+2$, while all other vertices in $G_s(\ell)$ have degree at most $3$, and in particular $S$ contains only vertices of degree $3$. 
    \begin{proposition}
        Let $s\geq 4$ be even and $\ell\in \mathbb{N}$.
        Let $G:=G_s(\ell)$ from Construction~\ref{construction max degree}.
        All the following are true.
        \begin{enumerate}
            \item
                $\gamma_P(G)=s$,\label{item power domination number of G}
            \item
                $S$ and $K$ are minimum power dominating sets of $G$, and\label{item min dominating sets of G}
            \item
                For any $\ell>s^22^{s+1}$, we have $\expolq{G}{A}\prec\expolq{G}{S}$ for any set $A$ of size $s$ which contains a vertex from $K$.\label{item high degree vertices are not always good}
        \end{enumerate}
    \end{proposition}
    
    \begin{proof}
        For~\ref{item power domination number of G} and~\ref{item min dominating sets of G}, we first note that each of the $F_i$'s are forts, and $F_i$ along with its entrance is $R_i$.
        Since the $R_i$'s are pairwise vertex-disjoint, we need at least $s$ vertices in any power dominating set to intersect with each fort, so $\gamma_P(G)\geq s$, while on the other hand it is easy to verify that $S$ and $K$ are power dominating sets of size $s$, and indeed any column forms a power dominating set.
    
        For~\ref{item high degree vertices are not always good}, let $\ell>s^22^{s+1}$.
        Towards applying Proposition~\ref{lemma sums imply dominating polynomial}, for $k\in [s]$ and $X\in \binom{V(G)}{s}$, let
        \[
            \lambda_k(X):=\sum_{X'\in \binom{X}{k}}\sizeobs{G}{X'}.
        \]
        We note that since $L_i\subseteq\obs{G}{v_{i,2s}}$ for each $i$, we have that for any $k\in [s]$,
        \[
            \lambda_k(S)\geq \binom{s}{k}k\ell.
        \]
        On the other hand, for any $A\in \binom{V(G)}{s}$, if there exists some $A'\in \binom{A}{k}$ such that $\obs{G}{A'}$ does not contain at least $k$ of the sets $L_i$, then
        \[
            \lambda_k(A)\leq\binom{s}{k}(k\ell+2s^2)-\ell,
        \]
        and so
        \[
            \lambda_k(S)-\lambda_k(A)\geq \binom{s}{k}2s^2-\ell=\binom{s}{k}2s^2-s^22^{s+1}>0,
        \]
        where we use that $\binom{s}{k}<2^s$ for any $s\geq k\geq 1$.
        Thus, our goal will be to show that for any $A\in \binom{V(G)}{s}$ which contains a vertex from $K$ has such a set $A'$ for each $k\in [s]$.
        Fix such an $A$ and $k$.
    
        \textbf{Case 1:} $k=1$.
        We note that the only vertices $v\in V(G)$ which have the property that $L_i\subseteq\obs{G}{v}$ are vertices in $S$, so since $A\neq S$, we are done.
    
        \textbf{Case 2:} $k=s$.
        Since $S$ is a power dominating set, we have trivially that $\lambda_s(S)=|V(G)|\geq \lambda_s(A)$.
    
        \textbf{Case 3:} $2\leq k\leq s-1$ and there exists some $i'\in [s]$ such that $|R_{i'}\cap A|\geq 2$.
        Let $A'\in \binom{A}{k}$ be a set such that $|R_{i'}\cap A'|\geq 2$.
        Then, there exists at most $k-1$ choices of $i\in [s]$ such that $R_i\cap A'\neq \emptyset$.
        Since each $F_i$ is a fort, the entrance of the fort is completely contained in $R_i$ and $L_i\subseteq F_i$, this implies that $A'$ can observe at most $k-1$ of the $L_i$'s.
    
        \textbf{Case 4:} $2\leq k\leq s-1$ and $|R_i\cap A|=1$ for all $i\in [s]$.
        Assume without loss of generality that $v_{1,2}\in A\cap K$, and let $A'\in \binom{A}{k}$ be a set with $v_{1,2}\in A'$.
        We claim that $L_1\not\in\obs{G}{A'}$.
        Instead, assume to the contrary that $L_1\in\obs{G}{A'}$, and let us work backwards.
        At some point $v_{1,2s}$ must have forced $v_{1,2s+1}$, which means that $v_{1,2s-1}$ must have already been observed prior to this, and $v_{1,2s-1}$ was not forced by $v_{1,2s}$.
        In particular, since $\mathrm{deg}(v_{1,2s})=2$, $v_{1,2s-1}$ must have forced $v_{1,2s-2}$.
        In this case, we can assume without loss of generality that actually $v_{1,2s-1}$ forced $v_{1,2s}$.
        Following the same logic, we can assume without loss of generality that $v_{1,2s-4}$ forced $v_{1,2s-3}$, which forced $v_{1,2s-2}$, and continuing along in this manner, we eventually find that we can assume without loss of generality that $v_{1,2}$ forced $v_{1,3}$, which ultimately lead to $L_1$ being observed.
    
        Let $i'\in [s]$ be such that $A'\cap R_{i'}=\emptyset$, and let $j'\in [\eta]$, $j'\geq 4$ denote the index such that $v_{1,j'}v_{i',j'}\in E(G)$ (such a column must exist since $\mathcal{M}$ is a $1$-factorization of $K_s$).
        By the preceding discussion, we must have that $v_{1,j'}$ forces $v_{1,j'+1}$, but the only other neighbors of $v_{i',j'}$ are in $F_{i'}$, a fort whose entrance and the entire fort is contained in $R_{i'}$, disjoint from $A'$.
        Thus $v_{i',j'}$ cannot be observed, and so $v_{1,j'}$ cannot force $v_{1,j'+1}$, a contradiction. 
    \end{proof}

\section{Acknowledgements}
    This project was sponsored, in part, by the Air Force Research Laboratory via the Autonomy Technology Research Center, University of Dayton, and Wright State University.
    This research was also supported by Air Force Office of Scientific Research award 23RYCOR004.

\bibliographystyle{plain}
\bibliography{bib}

\begin{thebibliography}{10}

\bibitem{AK2022}
S.~Anderson and K.~Kuenzel.
\newblock Power domination in cubic graphs and {C}artesian products.
\newblock {\em Discrete Math.}, 345(11):Paper No. 113113, 10, 2022.

\bibitem{BFFFHV2018}
K.~Benson, D.~Ferrero, M.~Flagg, V.~Furst, L.~Hogben, and V.~Vasilevska.
\newblock Nordhaus-{G}addum problems for power domination.
\newblock {\em Discrete Appl. Math.}, 251:103--113, 2018.

\bibitem{Bernstein1912}
Sergei Bernstein.
\newblock Démonstration du théorème de weierstrass fondeé sur le calcul des
  probabilités.
\newblock {\em Communications of the Kharkov Mathematical Society}, 13:1--2,
  1912.

\bibitem{BBFK2023}
B.~Bjorkman, Z.~Brennan, M.~Flagg, and J.~Koch.
\newblock Power domination with random sensor failure.
\newblock {\em arXiv:2312.12259}, 2023.

\bibitem{BCF2023}
B.~Bjorkman, E.~Conrad, and M.~Flagg.
\newblock An introduction to pmu-defect-robust power domination: Bounds,
  bipartites, and block graphs.
\newblock {\em arXiv:2312.07377}, 2023.

\bibitem{BBEFFH2019}
C.~Bozeman, B.~Brimkov, C.~Erickson, D.~Ferrero, M.~Flagg, and L.~Hogben.
\newblock Restricted power domination and zero forcing problems.
\newblock {\em J. Comb. Optim.}, 37(3):935--956, 2019.

\bibitem{B2010}
R.~Brualdi.
\newblock {\em Introductory combinatorics}.
\newblock Pearson Prentice Hall, Upper Saddle River, NJ, fifth edition, 2010.

\bibitem{BH2005}
D.~Brueni and L.~Heath.
\newblock The {PMU} placement problem.
\newblock {\em SIAM J. Discrete Math.}, 19(3):744--761, 2005.

\bibitem{DMKS2008}
P.~Dorbec, M.~Mollard, S.~Klav\v zar, and S.~\v~Spacapan.
\newblock Power domination in product graphs.
\newblock {\em SIAM J. Discrete Math.}, 22(2):554--567, 2008.

\bibitem{FH2018}
C.~Fast and I.~Hicks.
\newblock Effects of vertex degrees on the zero-forcing number and propagation
  time of a graph.
\newblock {\em Discrete Appl. Math.}, 250:215--226, 2018.

\bibitem{HHHH2002}
T.~Haynes, S.~Hedetniemi, S.~Hedetniemi, and M.~Henning.
\newblock Domination in graphs applied to electric power networks.
\newblock {\em SIAM J. Discrete Math.}, 15(4):519--529, 2002.

\bibitem{LMW2020}
C.~Lu, R.~Mao, and B.~Wang.
\newblock Power domination in regular claw-free graphs.
\newblock {\em Discrete Appl. Math.}, 284:401--415, 2020.

\bibitem{PCW2010}
K.~Pai, J.~Chang, and Y.~Wang.
\newblock Restricted power domination and fault-tolerant power domination on
  grids.
\newblock {\em Discrete Appl. Math.}, 158(10):1079--1089, 2010.

\bibitem{PAHA2024}
S.~Prabhu, A.~Arulmozhi, M.~Henning, and M.~Arulperumjothi.
\newblock Power domination and resolving power domination of fractal cubic
  network.
\newblock {\em arXiv:2407.01935}, 2024.

\bibitem{YW2022}
W.~Yang and B.~Wu.
\newblock Disproofs of three conjectures on the power domination of graphs.
\newblock {\em Discrete Appl. Math.}, 307:62--64, 2022.

\end{thebibliography}

\newpage
\appendix
\section{General Definitions and Notation}\label{sec:basic_definitions}
        For a given finite set $A$, the \emph{order} of $A$, denoted $|A|$, is the number of elements in $A$.
        The \emph{natural numbers}, denoted $\mathbb{N}$, is the set $\{0,1,2,\ldots\}$.
        
        A \emph{graph} is a set of vertices, $V(G)$, together with a set of edges, $E(G)$, consisting of unordered pairs of distinct vertices.
        Let $u$ and $v$ be vertices in $V(G)$.
        The edge between $u$ and $v$ will be written $uv$.
        If $uv \in E(G)$, then $u$ and $v$ are \emph{neighbors}.
        The \emph{degree} of $u$, denoted $\deg(u)$, is the number of neighbors $u$ has.
        The \emph{closed neighborhood of $u$}, denoted $N[u]$, is the set of neighbors of $u$ along with $u$ itself.
        In a similar way, the closed neighborhood of $S \subseteq V(G)$, denoted $N[S]$, is the set $\displaystyle\cup_{u \in S}N\left[u\right]$.

        Let $n \in \mathbb{N}$.
        The \emph{complete graph on $n$ vertices} or \emph{clique on $n$ vertices}, denoted $K_n$, is a graph where $|V(G)| = n$ and $E(G) = \binom{V(G)}{2}$.
        The \emph{empty graph on $n$ vertices} or \emph{independent graph on $n$ vertices}, denoted $\overline{K_n}$, is a graph where $|V(G)| = n$ and $E(G) = \emptyset$.
        The \emph{path on $n$ vertices}, denoted $P_n$, is a graph where $V(G)$ can be labeled $\{v_i \setgivensymbol i \in \mathbb{N}, i < n\}$ such that $E(G) = \{v_jv_{j+1} \setgivensymbol j \in \mathbb{N}, j<n-1\}$.
        The \emph{length of a path} is the number of edges it contains.
        The \emph{cycle on $n$ vertices}, denoted $C_n$, is a graph where $V(G)$ can be labeled as in $P_n$ such that $E(G) = E(P_n) \cup \{v_0v_n\}$.

        Given a graph $G$, the graph $H$ is a \emph{subgraph} of $G$, written $H \subseteq G$, if $V(H) \subseteq V(G)$ and $E(H) \subseteq E(G)$.
        The graph $H$ is an \emph{induced subgraph} of $G$ if $V(H) \subseteq V(G)$ and $E(H) = \{ uw \in E(G) \setgivensymbol u,w\in V(H)\}$.
        A subgraph $H$ is a \emph{spanning subgraph} of $G$ if $V(H) = V(G)$ and $V(H) \subseteq V(G)$.
        A \emph{path in $G$} is a subgraph $H\subseteq G$ where $H$ is path graph.
        A graph is \emph{connected} if there exists a path in $G$ between every distinct $v,w \in V(G)$.
        A graph is \emph{bipartite} if $V(G)$ can be partitioned into disjoint sets $A$ and $B$ such that no edge exists between vertices in $A$ nor between vertices in $B$.
        A \emph{$1$-factorization} of a graph is a partitioning of the edge set into perfect matchings.

        Given a graphs $G$ and $H$ with $uv \in E(G)$, $x \not \in V(G)$, $w \in V(H)$, $V(G) \cap V(H) = \emptyset$, define the following graph operations.
        To \emph{subdivide the edge $uv$} is to make the graph $G'$ with $V(G') = V(G) \cup \{x\}$ and $E(G') = \left(E(G) \setminus \{uv\} \right)\cup \{ux, xv\}$.
        The \emph{disjoint union of $G$ and $H$}, denoted $G \sqcup H$, is a graph operation resulting in a graph $G'$ with $V(G') = V(G) \cup V(H)$ and $E(G') = E(G) \cup E(H)$.
        To \emph{append $G$ to $H$ at $u$ and $w$} is to make the graph $G'$ with $V(G') = G \sqcup H$ and $E(G') = E(G \sqcup H) + \{uw\}$. 
        For some $S \subseteq V(G)$ and $T \subseteq V(H)$ such that $|S| = |T|$, to \emph{identify} the vertices in $S$ with the vertices in $T$ is to make the graph $G'$ with $V(G') = V(G) \cup \left( V(H) \setminus T \right)$ and $E(G') = E(G) \cup \{uv \setgivensymbol u,v \in V(H), u,v \not \in T\} \cup \{wf(w) \setgivensymbol w \in S\}$ for some bijective function $f:S \to T$.

\end{document}